\documentclass[11pt]{article}

\usepackage{enumerate}
\usepackage{amsmath}
\usepackage{amssymb,latexsym}
\usepackage{amsthm}
\usepackage{color}
\usepackage{graphicx}
\usepackage{hyperref}
\usepackage{amscd}

\DeclareMathOperator{\Tr}{Tr}
\DeclareMathOperator{\N}{N}

\title{A Carlitz type result for linearized polynomials}
\author{Bence Csajb\'ok, Giuseppe Marino, Olga Polverino\thanks{\textcolor{black}{The
research  was supported by the Italian National Group for Algebraic and Geometric Structures and their Applications (GNSAGA
- INdAM). The first author is supported by the J\'anos Bolyai Research Scholarship of the Hungarian Academy of Sciences. The first author acknowledges the support of OTKA Grant No. K 124950.}}}
\date{}

\newcommand{\F}{{\mathbb F}}

\newcommand{\la}{\langle}
\newcommand{\ra}{\rangle}

\newcommand{\G}{\mathrm{\Gamma L}}

\newtheorem{theorem}{Theorem}[section]

\newtheorem{lemma}[theorem]{Lemma}
\newtheorem{corollary}[theorem]{Corollary}

\newtheorem{proposition}[theorem]{Proposition}

\newtheorem{remark}[theorem]{Remark}

\DeclareMathOperator{\PG}{{PG}}
\DeclareMathOperator{\AG}{{AG}}

\begin{document}
\maketitle

\begin{abstract}
%Let $f$ and $g$ be two $q$-polynomials over $\F_{q^n}$.
For an arbitrary $q$-polynomial $f$ over $\F_{q^n}$ we study the problem of finding those $q$-polynomials $g$ over $\F_{q^n}$ for which the image sets of $f(x)/x$ and $g(x)/x$ coincide. For $n\leq 5$ we provide sufficient and necessary conditions and then apply our result to study maximum scattered linear sets of $\PG(1,q^5)$.
\end{abstract}

%\bigskip
%{\it AMS subject classification:}

%\bigskip
%{\it Keywords:}

\section{Introduction}

Let $\F_{q^n}$ denote the finite field of $q^n$ elements where $q=p^h$ for some prime $p$. For $n>1$ and $s \mid n$ the trace and norm over $\F_{q^s}$ of elements of $\F_{q^n}$ are defined as $\Tr_{q^n/q^s}(x)=x+x^{q^s}+\ldots+x^{q^{n-s}}$ and $\N_{q^n/q^s}(x)=x^{1+q^s+\ldots+q^{n-s}}$, respectively. When $s=1$ then we will simply write $\Tr(x)$ and $\N(x)$. Every function $f \colon \F_{q^n} \rightarrow \F_{q^n}$ can be given uniquely as a polynomial with coefficients in $\F_{q^n}$ and of degree at most $q^n-1$.
The function $f$ is $\F_q$-linear if and only if it is represented by a \emph{$q$-polynomial}, that is,
\begin{equation}
\label{fff}
f(x)=\sum_{i=0}^{n-1}a_i x^{q^i}
\end{equation}
with coefficients in $\F_{q^n}$. Such polynomials are also called \emph{linearized}.
%When $f$ is linearized, the following corollary of Theorem \ref{Carlitz} holds:
If $f$ is given as in \eqref{fff}, then its adjoint (w.r.t. the symmetric non-degenerate bilinear form defined by $\la x, y\ra=\Tr(xy)$) is
\[\hat f(x):=\sum_{i=0}^{n-1}a_i^{q^{n-i}} x^{q^{n-i}},\]
i.e. $\Tr(xf(y))=\Tr(y\hat f(x))$ for any $x,y\in\F_{q^n}$.

The aim of this paper is to study what can be said about two $q$-polynomials $f$ and $g$ over $\F_{q^n}$ if they satisfy
%Given two $q$-polynomials $f$ and $g$ over $\F_{q^n}$ such that
\begin{equation}
\label{start}
Im\,\left(\frac{f(x)}{x}\right)=Im\,\left(\frac{g(x)}{x}\right)\ \footnote{By $Im\,(f(x)/x)$ we mean the image of the rational function $f(x)/x$, i.e. $\{f(x)/x \colon x\in \F_{q^n}^*\}$.}.
\end{equation}

For a given $q$-polynomial $f$, the equality \eqref{start} clearly holds with $g(x)=f(\lambda x)/\lambda$ for each $\lambda\in \F_{q^n}^*$.
A less obvious choice when \eqref{start} holds is when $g(x)=\hat f(\lambda x)/\lambda$, see \cite[Lemma 2.6]{BGMP2015} and the first part of \cite[Section 3]{CSMP2016}.

\medskip

When one of the functions in \eqref{start} is a monomial then the answer to the question posed above follows from McConnel's generalization \cite[Theorem 1]{McConnel} of a result due to Carlitz \cite{Carlitz} (see also Bruen and Levinger \cite{BL}).

\begin{theorem}{\cite[Theorem 1]{McConnel}}
\label{Carlitz}
Let $p$ denote a prime, $q=p^h$, and $1<d$ a divisor of $q-1$. Also, let $F \colon \F_{q} \rightarrow \F_{q}$ be a function such that $F(0)=0$ and $F(1)=1$.
Then
\[\left(F(x)-F(y)\right)^{\frac{q-1}{d}}=\left(x-y\right)^{\frac{q-1}{d}}\]
for all $x,y\in \F_{q}$ if and only if $F(x)=x^{p^j}$ for some $0 \leq j < h$ and $d\mid p^j-1$.
\end{theorem}

Indeed, when the function $F$ of Theorem \ref{Carlitz} is $\F_q$-linear, we easily get the following corollary (see Section \ref{BPR} for the proof, or   \cite[Corollary 1.4]{Praha} for the case when $q$ is an odd prime).

\begin{corollary}
\label{norm}
Let $g(x)$ and $f(x)=\alpha x^{q^k}$, $q=p^h$, be $q$-polynomials over $\F_{q^n}$ such that
\begin{equation}\label{form:im}
Im\,\left(\frac{f(x)}{x}\right)=Im\,\left(\frac{g(x)}{x}\right).
\end{equation}
Denote $\gcd(k,n)$ by $t$. Then $g(x)=\beta x^{q^s}$ with $\gcd(s,n)=t$ for some $\beta$ with $\N_{q^n/q^t}(\alpha)=\N_{q^n/q^t}(\beta)$. 
\end{corollary}

\noindent Another case for which we know a complete answer to our problem is when $f(x)=\Tr(x)$.

\begin{theorem}[{\cite[Theorem 3.7]{CSMP2016}}]
\label{trace}
Let $f(x)=\Tr(x)$ and let $g(x)$ be a $q$-polynomial over $\F_{q^n}$ such that
\[Im\,(f(x)/x)=Im \,(g(x)/x).\]
Then $g(x)=\Tr(\lambda x)/\lambda$ for some $\lambda \in \F_{q^n}^*$.
\end{theorem}

\noindent Note that in Theorem \ref{trace} we have $\hat f(x) = f(x)$ and the only solutions for $g$ are $g(x)=f(\lambda x)/\lambda$, while in Corollary \ref{norm} we have (up to scalars) $\varphi(n)$ different solutions for $g$, where $\varphi$ is the Euler's totient function.

The problem posed in \eqref{start} is also related to the study of the directions determined by an additive function. Indeed, when $f$ is additive, then
\[Im\,(f(x)/x)=\left\{\frac{f(x)-f(y)}{x-y} \colon x\neq y,\, x,y\in \F_{q^n}\right\},\]
is the \emph{set of directions} determined by the graph of $f$, i.e. by the point set ${\mathcal G}_f:=\{(x,f(x)) \colon x\in \F_{q^n}\}\subset \AG(2,q^n)$. Hence, in this setting, the problem posed in \eqref{start} corresponds to finding the $\F_q$-linear functions whose graph determines the same set of directions. %we want to get information on linear functions determining the same set of directions.
%The size of ${\mathcal D}_f$ (for any $f$, not necessarily additive) was studied extensively.
The maximum cardinality of the image set of $f(x)/x$, where $f$ is a $q$-polynomial
over $\F_{q^n}$, is $(q^n-1)/(q-1)$, whereas, if $\F_q$ is the maximum field of linearity of $f$, then by \cite{BBBSSz, B2003, Cs2017} $Im\,(f(x)/x)$ contains at least $q^{n-1}+1$ elements, hence
\begin{equation}
\label{directions}
q^{n-1}+1\leq |Im\,(f(x)/x)| \leq \frac{q^{n}-1}{q-1}.
\end{equation}

The classical examples which show the sharpness of these bounds are the monomial functions $x^{q^s}$, with $\gcd (s,n)=1$, and the $\Tr(x)$ function.
However, these bounds are also achieved by other polynomials which are not "equivalent" to these examples (see Section 2 for more details).
\bigskip

Two $\F_q$-linear polynomials $f(x)$ and $h(x)$ of $\F_{q^n}[x]$ are \emph{equivalent} if the two graphs ${\mathcal G}_f$ and ${\mathcal G}_h$ are equivalent under tha action of the group $\Gamma {\mathrm L}(2,q^n)$, i.e. if there exists an element $\varphi\in\Gamma {\mathrm L}(2,q^n)$ such that ${\mathcal G}_f^\varphi={\mathcal G}_h$. In such a case, we say that $f$ and $h$ are \emph{equivalent} (via $\varphi$) and we write $h=f_\varphi$. It is easy to see that in this way we defined an equivalence relation on the set of $q$-polynomials over $\F_{q^n}$. If $f$ and $g$ are two $q$-polynomials such that $Im\,(f(x)/x)=Im\,(g(x)/x)$, then $Im\,(f_\varphi(x)/x)=Im\,(g_\varphi(x)/x)$ for any admissible $\varphi\in\G(2,q^n)$ (see Proposition \eqref{imtrans}). This means that the problem posed in \eqref{start} can be investigated up to equivalence.

\bigskip

For $n\leq 4$, the only solutions for $g$ in Problem \eqref{start} are the trivial ones, i.e. either $g(x)=f(\lambda x)/x$ or $g(x)=\hat f(\lambda x)/x$ (cf. Theorem \ref{n=4}).

For the case $n=5$, in Section \ref{section-proof}, we prove the following main result.

\begin{theorem}
\label{main}
Let $f(x)$ and $g(x)$ be two $q$-polynomials over $\F_{q^5}$, with maximum field of linearity $\F_q$, such that $Im\,(f(x)/x)=Im\,(g(x)/x)$.
Then either there exists $\varphi\in \G(2,q^5)$ such that $f_\varphi(x)=\alpha x^{q^i}$ and $g_\varphi(x)=\beta x^{q^j}$ with $\N(\alpha)=\N(\beta)$ for some
$i,j\in \{1,2,3,4\}$, or there exists $\lambda\in \F_{q^5}^*$ such that $g(x)=f(\lambda x)/\lambda$ or $g(x)=\hat{f}(\lambda x)/\lambda$.
\end{theorem}

Finally, the relation between $Im\,(f(x)/x)$ and the linear sets of rank $n$ of the projective line $\PG(1,q^n)$ will be pointed out in Section \ref{geom}.
As an application of Theorem \ref{main} we get a criteria of $\mathrm{P}\Gamma\mathrm{L}(2,q^5)$-equivalence for linear sets in $\PG(1,q^5)$ and this allows us to prove that the family of (maximum scattered) linear sets of rank $n$ of size $(q^n-1)/(q-1)$ in $\PG(1,q^n)$ found by Sheekey in \cite{Sh} contains members which are not-equivalent to the previously known linear sets of this size.

\section{Background and preliminary results}
\label{BPR}

Let us start this section by the following immediate corollary of \eqref{directions}.

\begin{proposition}
	\label{fieldoflinearity}
	If $Im\,(f(x)/x)=Im\,(g(x)/x)$ for two $q$-polynomials $f$ and $g$ over $\F_{q^n}$, then their maximum field of linearity coincide.
\end{proposition}
\begin{proof}
	Let $\F_{q^m}$ and $\F_{q^k}$ be the maximum fields of linearity of $f$ and $g$, respectively. Suppose to the contrary $m<k$.
	Then $|Im\,(g(x)/x)|\leq (q^n-1)/(q^k-1)<q^{n-k+1}+1\leq q^{n-m}+1 \leq |Im\,(f(x)/x)|$, a contradiction.
\end{proof}

Now we are able to prove Corollary \ref{norm}.

\begin{proof}
The maximum field of linearity of $f(x)$ is $\F_{q^t}$, thus, by Proposition \ref{fieldoflinearity}, $g(x)$ has to be a $q^t$-polynomial as well. 
Then for $t>1$ the result follows from the $t=1$ case (after substituting $q$ for $q^t$ and $n/t$ for $n$) and hence we can assume that $f(x)$ and $g(x)$ are strictly $\F_q$-linear. By \eqref{form:im}, we note that $g(1)=\alpha z_0^{q^k-1}$, for some $z_0\in\F_{q^n}^*$. Let $F(x):=g(x)/g(1)$, then $F$ is a $q$-polynomial over $\F_{q^n}$, with $F(0)=0$ and $F(1)=1$. Also, from \eqref{form:im}, for each $x\in\F_{q^n}^*$ there exists $z\in\F_{q^n}^*$ such that
\[\frac{F(x)}{x}=\left(\frac{z}{z_0}\right)^{q^k-1}.\] This means that for each $x \in\F_{q^n}^*$ we get $\N\left(\frac{F(x)}{x}\right)=1$. From Theorem \ref{Carlitz} (with $d=q-1$) it follows that $F(x)=x^{p^j}$ for some $0\leq j<nh$. 
Then Proposition \ref{fieldoflinearity} yields $p^j=q^s$ with $\gcd(s,n)=1$. We get the first part of the statement by putting $\beta:=f(1)$. Again from \eqref{form:im} it follows that $\N(\alpha)=\N(\beta)$.
\end{proof}

%One might ask what happens if in Theorem \ref{norm} we replace $x^q$ by $\Tr(x)$.

%\begin{theorem}[{\cite[Theorem 3.7]{CSMP2016}}]
%\label{trace}
%Put $g(x)=\Tr(x)$ and let $f(x)$ be a $q$-polynomial over $\F_{q^n}$ such that $Im\,(f(x)/x)=Im \,(g(x)/x)$. Then $f(x)=\Tr(\lambda x)/\lambda$ for some $\lambda \in \F_{q^n}^*$.
%\end{theorem}

Let  $f$ and $g$ be two equivalent $q$-polynomials over $\F_{q^n}$ via the element  $\varphi\in\Gamma{\mathrm L}(2,q^n)$ represented by the invertible matrix $\begin{pmatrix}
a&b\\
c&d\\
\end{pmatrix}$ and with companion automorphism $\sigma$ of $\F_{q^n}$. Then
\begin{equation}\label{form:equiv}
\left\{
\begin{pmatrix}
x\\
g(x)
\end{pmatrix}
\colon x\in \F_{q^n}
\right\}=
\left\{
\begin{pmatrix}
a&b\\
c&d\\
\end{pmatrix}
\begin{pmatrix}
x^\sigma\\
f(x)^\sigma
\end{pmatrix}
\colon x \in \F_{q^n}\right\}.
\end{equation}
By \eqref{form:equiv} we get that the function $k_f(x):=ax^\sigma+bf(x)^\sigma$ is invertible and, with $h_f(x):=cx^\sigma+df(x)^\sigma$, it follows that  $g(x)=h_f(k_f^{-1}(x))$. Hence, we have proved the following proposition.

\begin{proposition}\label{fvarphi}
Let $f$ and $g$ be $q$-polynomials over $\F_{q^n}$ such that $g=f_\varphi$ for some $\varphi\in\G(2,q^n)$ with linear part represented by $\left(
\begin{array}{cc}
a & b \\
c & d \\
\end{array}
\right)$ and with companion automorphism $\sigma$. Let $k_f(x)=ax^\sigma+bf(x)^\sigma$ and
$h_f(x)=cx^\sigma+df(x)^\sigma$, then $k_f$ is invertible and $g(x)=f_\varphi(x)=h_f(k_f^{-1}(x))$.
\end{proposition}

From \eqref{form:equiv} it is also clear that
\begin{equation}
\label{map}
Im\,\left(\frac{f_\varphi(x)}x\right)=\left\{\frac{c+dz^\sigma}{a+bz^\sigma}
\colon z \in Im\,\left(\frac{f(x)}x\right)\right\}
\end{equation}
and hence 
\begin{equation}
\label{size}
|Im\,(f_\varphi(x)/x)|=|Im\,(f(x)/x)|.
\end{equation}

From Equations \eqref{size} and \eqref{directions} the next result easily follows.
\begin{proposition}
\label{prop:lin}
	If two $q$-polynomials over $\F_{q^n}$ are equivalent, then their maximum field of linearity coincide.\qed
\end{proposition}

Note that $|Im\,(g(x)/x)|=|Im\,(f(x)/x)|$ does not imply the equivalence of $f$ and $g$. In fact, in the last section we will list the known examples of $q$-polynomials $f$ which are not equivalent to monomials but the size of $Im\,(f(x)/x)$ is maximal. To find such functions was also proposed in \cite{Praha} and, as it was observed by Sheekey, they determine certain MRD-codes \cite{Sh}.

The following results will be useful later in the paper.

\begin{proposition}
\label{imtrans}
If $Im\, (f(x)/x)=Im\, (g(x)/x)$ for some $q$-polynomials over $\F_{q^n}$, then $Im\, (f_\varphi(x)/x)=Im\, (g_\varphi(x)/x)$ holds for each admissible \footnote{An element $\varphi\in \G(2,q^n)$ represented by the invertible matrix $\begin{pmatrix}
a&b\\
c&d\\
\end{pmatrix}$ and with companion automorphism $\sigma$ of $\F_{q^n}$ is said to be \emph{admissible} w.r.t. a given $q$-polynomial $f$ if either $b=0$ or $-(a/b)^{\sigma^{-1}}\notin Im\, (f(x)/x)$.} $\varphi\in \G(2,q^n)$.
\end{proposition}
\begin{proof}
From $Im\,(f(x)/x)=Im\, (g(x)/x)$ it follows that any $\varphi\in \G(2,q^n)$ admissible w.r.t. $f$ is admissible w.r.t. $g$ as well. Hence $k_f$ and $k_g$ are both invertible and we may construct $f_\varphi$ and $g_\varphi$ as indicated in Proposition \ref{fvarphi}. The statement now follows from Equation \eqref{map}.
\end{proof}

%The next result is easy to prove.

\begin{proposition}
\label{lambda0}
Let $f$ and $g$ be $q$-polynomials over $\F_{q^n}$ and take some $\varphi\in \G(2,q^n)$ with companion automorphism $\sigma$. Then 
$g_\varphi(x)=f_\varphi(\lambda^\sigma x)/\lambda^\sigma$
for some $\lambda\in \F_{q^n}^*$ if and only if $g(x)=f(\lambda x)/\lambda$.
%Let $f$ be a $q$-polynomial over $\F_{q^n}$ and put $g(x)=f(\lambda x)/\lambda$ for some $\lambda\in \F_{q^n}^*$.
%Then for each $\varphi\in \G(2,q^n)$ with companion automorphism $\sigma$ we have $g_\varphi(x)=f_\varphi(\lambda^\sigma x)/\lambda^\sigma$. Conversely, if there exists an element $\varphi\in \G(2,q^n)$ such that $g_\varphi(x)=f_\varphi(\lambda^\sigma x)/\lambda^\sigma$ for some $\lambda\in \F_{q^n}^*$, then $g(x)=f(\lambda x)/\lambda$
\end{proposition}

\begin{proof}
First we prove the "if" part. Since $g(x)=f(\lambda x)/\lambda=(\omega_{1/\lambda}\circ f \circ \omega_\lambda)(x)$, where $\omega_\alpha$ denotes the scalar map $x\in\F_{q^n}\mapsto\alpha x\in\F_{q^n}$, direct computations show that $h_g=\omega_{1/\lambda^\sigma}\circ h_f\circ\omega_\lambda$ and $k_g=\omega_{1/\lambda^\sigma}\circ k_f\circ\omega_\lambda$. Then $g_\varphi=\omega_{1/\lambda^\sigma}\circ f_\varphi\circ\omega_{\lambda^\sigma}$ and the first part of the statement follows. The "only if" part follows from the "if" part applied to $g_\varphi(x)=f_\varphi(\lambda^\sigma x)/\lambda^\sigma$ and $\varphi^{-1}$; and from $(f_\varphi)_{\varphi^{-1}}=f$ and $(g_\varphi)_{\varphi^{-1}}=g$. 
\end{proof}

%\begin{proposition}
%\label{adjointlambda0}
%\begin{proposition}
%Let $f$ and $g$ be $q$-polynomials over $\F_{q^n}$ such that $g=f^\varphi$ for some $\varphi\in\G(2,q^n)$ with linear part represented by $\left(
%\begin{array}{cc}
%a & b \\
%c & d \\
%\end{array}
%\right)$ and with companion automorphism $\sigma$. Then
%\[\]
%\end{proposition}

Next we summarize what is known about Problem \eqref{start} for $n\leq 4$.

\begin{theorem}
	\label{n=4}
	Suppose $Im\, (f(x)/x)=Im\, (g(x)/x)$ for some $q$-polynomials over $\F_{q^n}$, $n\leq 4$, with maximum field of linearity $\F_q$. Then there exist $\varphi\in \mathrm{GL}(2,q^n)$ and $\lambda \in \F_{q^n}^*$ such that the following holds.
	\begin{itemize}
		\item If $n=2$ then $f_\varphi(x) = x^q$ and $g(x) = f(\lambda x)/\lambda$.
		\item If $n=3$ then either
		\[f_\varphi(x)=\Tr(x) \mbox{ and } g(x)=f(\lambda x)/\lambda\]
		or
		\[f_\varphi(x) = x^q \mbox{ and } g(x)=f(\lambda x)/\lambda \mbox{ or } g(x)=\hat{f}(\lambda x)/\lambda.\]
		\item  If $n=4$ then $g(x) = f(\lambda x)/\lambda$ or $g(x) = \hat{f}(\lambda x)/\lambda$.
	\end{itemize}
\end{theorem}
\begin{proof}
	In the $n=2$ case $f(x)=ax+bx^q$, $b\neq 0$. Then $\varphi:=\begin{pmatrix}
	1 & 0 \\
	-a/b & 1/b \\
\end{pmatrix}$ maps $f(x)$ to $x^q$. Then Proposition \ref{imtrans} and Corollary \ref{norm} give $g_\varphi(x)=f_\varphi(\mu x)/\mu$ and hence Proposition \ref{lambda0} gives $g(x)=f(\lambda x)/\lambda$ for some $\lambda\in\F_{q^n}$.
	If $n=3$ then according to \cite[Theorem 5]{LaVa2010} and \cite[Theorem 1.3]{CSMP2016}  there exists $\varphi\in \mathrm{GL}(2,q^3)$ such that either $f_\varphi(x)=\Tr(x)$ or $f_\varphi(x)=x^q$.
	In the former case Proposition \ref{imtrans} and Theorem \ref{trace} give $g_\varphi(x)=f_\varphi(\mu x)/\mu$ and the assertion follows from Proposition \ref{lambda0}. In the latter case Proposition \ref{imtrans} and Corollary \ref{norm} give $g_\varphi(x)=\alpha x^{q^i}$ where $i\in \{1,2\}$ and $\N(\alpha)=1$. If $i=1$, then $g_\varphi(x)=f_\varphi(\mu x)/\mu$ where $\mu^{q-1}=\alpha$ and the assertion follows from Proposition \ref{lambda0}. Let now $i=2$ and denote by 
	$\begin{pmatrix}
	A&B\\
	C&D\\
	\end{pmatrix}$ the matrix of $\varphi^{-1}$. Also, let $\Delta$ denote the determinant of this matrix. Since $\N(\alpha)=1$, direct computations show
	\[f(x)=(f_{\varphi})_{\varphi^{-1}}(x)=\frac{(A^{q+q^2}C+B^{q+q^2}D)x+A^{q^2}\Delta x^q-B^q\Delta x^{q^2}}{\N(A)+\N(B)},\]
	\[g(x)=(g_{\varphi})_{\varphi^{-1}}(x)=\frac{(A^{q+q^2}C+B^{q+q^2}D)x-B^{q^2}\Delta \alpha^{q^2+1}x^q+A^q\Delta\alpha x^{q^2}}{\N(A)+\N(B)},\]
	and hence $g(x)=\hat{f}(\lambda x)/\lambda$ for each $\lambda\in \F_{q^3}^*$ with  $\lambda^{q-1}=\Delta^{1-q}/\alpha^q$.
	
	The case $n=4$ is \cite[Proposition 4.2 ]{CSMP2016}.
\end{proof}

\begin{remark}
{\rm
Theorem \ref{n=4} yields that there is a unique equivalence class of $q$-polynomials, with maximum field of linearity $\F_q$, when $n=2$. For $n=3$ there are two non-equivalent classes and they correspond to the classical examples: $\Tr(x)$ and $x^q$. Whereas, for $n=4$, from \cite[Sec. 5.3]{CSMP2016} and \cite[Table p. 54]{BoPo2005}, there exist at least eight non-equivalent classes. The possible sizes for the sets of directions determined by these strictly $\F_q$-linear functions are $q^3+1$, $q^3+q^2-q+1$, $q^3+q^2+1$ and $q^3+q^2+q+1$ and each of them is determined by at least two non-equivalent $q$-polynomials. Also, by \cite[Theorem 3.4]{CSZ2017}, if $f$ is a $q$-polynomial over $\F_{q^4}$ for which the set of directions is of maximum size then $f$ is equivalent either to $x^q$ or to $\delta x^q+x^{q^3}$, for some $\delta\in\F_{q^4}^*$ with $N(\delta)\ne 1$ (see \cite{LP2001}).}
\end{remark}

\section{\texorpdfstring{Preliminary results about $\Tr(x)$ and the monomial $q$-polynomials over $\F_{q^5}$}{Preliminary results about Tr(x) and the monomial q-polynomials over Fq5}}

%For $\alpha\in \F_{q^n}$ and a divisor $h$ of $n$ we will denote by $\Tr_{q^n/q^h}(\alpha)$ the trace of $\alpha$ over the subfield $\F_{q^h}$, that is, $\Tr_{q^n/q^h}(\alpha)=\alpha+\alpha^{q^h}+\ldots+\alpha^{q^{n-h}}$.
%For $x\in \F_{q^5}$ the trace and norm of $x$ over $\F_q$ will be denoted by $\Tr(x)$ and $\N(x)$, respectively.
Let $q$ be a power of a prime $p$. We will need the following results.

\begin{proposition}
\label{trace5}
Let $f(x)=\sum_{i=0}^4a_ix^{q^i}$ and $g(x)=\Tr(x)$ be $q$-polynomials over $\F_{q^5}$. Then there is an element $\varphi\in\Gamma\mathrm{L}(2,q^5)$ such that $Im\,(f_\varphi(x)/x)=Im\,(g(x)/x)$ if and only if $a_1a_2a_3a_4\neq0$, $(a_1/a_2)^q=a_2/a_3$, $(a_2/a_3)^q=a_3/a_4$ and $\N(a_1)=\N(a_2)$.
\end{proposition}
\begin{proof}
Let $\varphi\in\Gamma\mathrm{L}(2,q^5)$ such that $Im\,(f_\varphi(x)/x)=Im\,(g(x)/x)$.  By Proposition \ref{prop:lin}, the maximum field of linearity of $f$ is $\F_q$ and by Theorem \ref{trace} there exists $\lambda\in\F_{q^5}^*$ such that $f_\varphi(x)=\Tr(\lambda x)/\lambda$.%According to \cite[Theorem 3.7]{CSMP2016} this happens if and only there exists $\mu\in \F_{q^5}^*$ such that ${\mathcal G}_g=\mu {\mathcal G}_f^{\varphi}$.
This is equivalent to say that there exist $a,b,c,d$, $ad-bc\neq 0$ and $\sigma=p^h$ such that
\[\left\{
\begin{pmatrix}
y\\
\Tr(y)
\end{pmatrix}
\colon y\in \F_{q^5}
\right\}=
\left\{
\begin{pmatrix}
a&b\\
c&d\\
\end{pmatrix}
\begin{pmatrix}
x^\sigma\\
f(x)^\sigma
\end{pmatrix}
\colon x \in \F_{q^5}\right\}.\]
Then $c x^\sigma+d f(x)^\sigma \in \F_q$ for each $x\in \F_{q^5}$. Let $z=x^{\sigma}$. Then
\[cz+d\sum_{i=0}^4a_i^{\sigma}z^{q^i}=c^q z^{q}+d^q\sum_{i=0}^4 a_i^{\sigma q}z^{q^{i+1}},\]
for each $z$. As polynomials of $z$ the left and right-hand sides of the above equation coincide modulo $z^{q^5}-z$ and hence comparing coefficients yield
\[c+da_0^\sigma = d^q a_4^{\sigma q},\]
\[d a_1^\sigma = c^q + d^q a_0^{\sigma q},\]
\[d a_{k+1}^\sigma=d^q a_k^{\sigma q},\]
for $k=1,2,3$.
If $d=0$, then $c=0$, a contradiction.
Since $d\neq 0$, if one of $a_1,a_2,a_3,a_4$ is zero, then all of them are zero and hence $f$ is $\F_{q^5}$-linear. This is not the case, so we have $a_1a_2a_3a_4\neq 0$.
Then the last three equations yield
\[\left(\frac{a_1}{a_2}\right)^q=\frac{a_2}{a_3},\]
\[\left(\frac{a_2}{a_3} \right)^q=\frac{a_3}{a_4},\]
and also $\N(a_1)=\N(a_2)$.

Now assume that the conditions of the assertion hold.
It follows that $a_3=a_2^{q+1}/a_1^q$ and $a_4=a_3^{q+1}/a_2^q=a_2^{q^2+q+1}/a_1^{q^2+q}$.
Let $\alpha_i=a_i/a_1$ for $i=0,1,2,3,4$. Then $\alpha_1=1$, $\N(\alpha_2)=1$, $\alpha_3=\alpha_2^{q+1}$ and $\alpha_4=\alpha_2^{1+q+q^2}$. We have $\alpha_2=\lambda^{q-1}$ for some $\lambda\in \F_{q^5}^*$.
If
\[
\begin{pmatrix}
a&b\\
c&d
\end{pmatrix}
=
\begin{pmatrix}
1&0\\
1-\lambda^{1-q^4}a_0/a_1& \lambda^{1-q^4}/a_1
\end{pmatrix},
\]
then
\[
\begin{pmatrix}
a&b\\
c&d
\end{pmatrix}
\begin{pmatrix}
x\\
f(x)
\end{pmatrix}
=\]
\[\begin{pmatrix}
x\\
x+\lambda^{1-q^4}x^q+\lambda^{q-q^4}x^{q^2}+\lambda^{q^2-q^4}x^{q^3}+\lambda^{q^3-q^4}x^{q^4}
\end{pmatrix}
=
\begin{pmatrix}
x\\
\Tr(x \lambda^{q^4})/\lambda^{q^4}
\end{pmatrix},\] i.e. $f_\varphi(x)=\Tr(\lambda^{q^4}x)/\lambda^{q^4}$, where $\varphi$ is defined by the matrix $\begin{pmatrix}
a&b\\
c&d
\end{pmatrix}$.
\end{proof}

\begin{proposition}
\label{pseudoalg}
Let $f(x)=\sum_{i=0}^4a_ix^{q^i}$, with $a_1a_2a_3a_4\neq 0$. 
%and $g(x)=x^{q^i}$, with $i\in\{1,2,3,4\}$, be two $q$-polynomials over $\F_{q^5}$.  
Then there is an element $\varphi\in\Gamma\mathrm{L}(2,q^5)$ such that $Im\,(f_\varphi(x)/x)=Im\,(x^{q-1})$ if and only if one of the following holds:
\begin{enumerate}
	\item $(a_1/a_2)^q=a_2/a_3$, $(a_2/a_3)^q=a_3/a_4$ and $\N(a_1)\neq \N(a_2)$, or
	\item $(a_4/a_1)^{q^2}=a_1/a_3$, $(a_1/a_2)^{q^2}=a_3/a_4$ and $\N(a_1)\neq \N(a_3)$.
\end{enumerate}
In both cases, if the condition on the norms does not hold, then $Im\,(f_\varphi(x)/x)=Im\,(\Tr(x)/x)$.
\end{proposition}
\begin{proof}
%$L_f$ is of pseudoregulus type if and only if there exists a semilinear map $\varphi$ such that $L_g=L_{U_f^\varphi}$.
We first note that the monomials $x^{q^i}$ and $x^{q^{5-i}}$ are equivalent via the map $\psi:=\begin{pmatrix}
0&1\\
1&0
\end{pmatrix}$. Hence, by Corollary \ref{norm}, the statement holds if and only if there exist $a,b,c,d$, $ad-bc\neq 0$, $\sigma=p^h$ and $i\in \{1,2\}$ such that
\begin{equation}
\label{pseudoxx}
\left\{
\begin{pmatrix}
y\\
y^{q^i}
\end{pmatrix}
\colon y\in \F_{q^5}
\right\}
=
\left\{
\begin{pmatrix}
a&b\\
c&d\\
\end{pmatrix}
\begin{pmatrix}
x^\sigma\\
f(x)^\sigma
\end{pmatrix}
\colon x \in \F_{q^5}\right\}
.
\end{equation}

If Condition 1 holds then let %$a_3=a_2^{q+1}/a_1^q$ and $a_4=a_3^{q+1}/a_2^q=a_2^{q^2+q+1}/a_1^{q^2+q}$.
$\alpha_j=a_j/a_1$ for $j=0,1,2,3,4$. So $\alpha_1=1$, $\N(\alpha_2)\neq 1$, $\alpha_3=\alpha_2^{q+1}$, $\alpha_4=\alpha_2^{1+q+q^2}$ and \eqref{pseudoxx} is satisfied with
\[
\begin{pmatrix}
a&b\\
c&d
\end{pmatrix}
=
\begin{pmatrix}
1&\alpha_2^{q^4}\\
\alpha_2^{1+q+q^2+q^3}&1
\end{pmatrix}
\begin{pmatrix}
1&0\\
-\alpha_0 & 1/a_1
\end{pmatrix}.
\]
%then
%\[
%\begin{pmatrix}
%a&b\\
%c&d
%\end{pmatrix}
%\begin{pmatrix}
%x\\
%f(x)
%\end{pmatrix}
%=\]
%\[ \begin{pmatrix}
%1&\alpha_2^{q^4}\\
%\alpha_2^{1+q+q^2+q^3}&1
%\end{pmatrix}
%\begin{pmatrix}
%x\\
%x^q+\alpha_2 x^{q^2}+\alpha_3 x^{q^3}+\alpha_4 x^{q^4}
%\end{pmatrix}
%=
%\begin{pmatrix}
%y\\
%y^q
%\end{pmatrix},
%\]
%with $y=x+\alpha_2^{q^4}(x^q+\alpha_2 x^{q^2}+\alpha_2^{q+1}x^{q^3}+\alpha_2^{1+q+q^2}x^{q^4})$.

If Condition 2 holds then let $\alpha_j=a_j/a_3$ for $j=0,1,2,3,4$. So $\alpha_3=1$, $\N(\alpha_1)\neq 1$, $\alpha_2=\alpha_1^{1+q+q^3}$,  $\alpha_4=\alpha_1^{1+q^3}$ and \eqref{pseudoxx} is satisfied with
\[
\begin{pmatrix}
a&b\\
c&d
\end{pmatrix}
=
\begin{pmatrix}
\alpha_1^{1+q+q^3+q^4}&1\\
1&\alpha_1^{q^2}
\end{pmatrix}
\begin{pmatrix}
1&0\\
-\alpha_0 & 1/a_3
\end{pmatrix}.
\]
%then
%\[
%\begin{pmatrix}
%a&b\\
%c&d
%\end{pmatrix}
%\begin{pmatrix}
%x\\
%f(x)
%\end{pmatrix}
%=\]
%\[\begin{pmatrix}
%\alpha_1^{1+q+q^3+q^4}&1\\
%1&\alpha_1^{q^2}
%\end{pmatrix}
%\begin{pmatrix}
%x\\
%\alpha_1x^q+\alpha_2x^{q^2}+\alpha_3x^{q^3}+\alpha_4x^{q^4}
%\end{pmatrix}
%=
%\begin{pmatrix}
%y\\
%y^{q^2}
%\end{pmatrix},
%\]
%with $y=x\alpha_1^{1+q+q^3+q^4}+\alpha_1x^q+\alpha_1^{q^3+q+1}x^{q^2}+x^{q^3}+\alpha_1^{q^3+1}x^{q^4}$.

Suppose now that \eqref{pseudoxx} holds and put $z=x^\sigma$. Then
\[(za+b\sum_{j=0}^4 a_j^\sigma z^{q^j})^{q^i}=cz+d\sum_{j=0}^4a_j^\sigma z^{q^j}\]
for each $z\in \F_{q^5}$ and hence, as polynomials in $z$, the left-hand side and right-hand side of the above equation coincide modulo $z^{q^5}-z$.
The coefficients of $z$, $z^{q^i}$ and $z^{q^k}$ with $i\in \{1,2\}$ and $k\in \{1,2,3,4\}\setminus \{i\}$ give
\[b^{q^i} a_{-i}^{\sigma q^i} =c+da_0^\sigma,\]
\[a^{q^i} +b^{q^i} a_0^{\sigma q^i}=d a_i^\sigma,\]
\[b^{q^i} a_{k-i}^{\sigma q^i}=d a_{k}^\sigma,\]
respectively, where the indices are considered modulo 5.
Note that $db\neq 0$ since otherwise also $a=c=0$ and hence $ad-bc=0$.
With $\{r,s,t\}=\{1,2,3,4\} \setminus \{i\}$, the last three equations yield:
\[\left(\frac{a_{r-i}}{a_{s-i}}\right)^{q^i}=\frac{a_r}{a_s},\]
\[\left(\frac{a_{s-i}}{a_{t-i}} \right)^{q^i}=\frac{a_s}{a_t}.\]
First assume $i=1$.
Then we have
\[\left(\frac{a_1}{a_2}\right)^q=\frac{a_2}{a_3} \quad \text{and} \quad \left(\frac{a_2}{a_3} \right)^q=\frac{a_3}{a_4}.\]
If $\N(a_1)=\N(a_2)$, from Proposition \ref{trace5} and Equation \eqref{size} it follows that $|Im\,(x^{q-1})|=|Im\,(\Tr(x)/x)|$, a contradiction.

Now assume $i=2$. Then we have $(a_4/a_1)^{q^2}=a_1/a_3$ and
\begin{equation}
\label{form1}
\left(\frac{a_1}{a_2} \right)^{q^2}=\frac{a_3}{a_4}.
\end{equation}
Multiplying these two equations yields $a_4^{q^2+1}=a_1a_2^{q^2}$ and hence %$a_3^{q^2+1}a_2^{q^4+q^2}/a_1^{q^4+q^2}=a_1a_2^{q^2}$.
%\[\left(\frac{a_4}{a_2} \right)^{q^2}=\frac{a_1}{a_4}.\]
%It follows that $a_4=a_3 a_2^{q^2}/a_1^{q^2}$,  and hence $a_3^{q^2+1}a_2^{q^4+q^2}/a_1^{q^4+q^2}=a_1a_2^{q^2}$.
%After rearranging we get
\begin{equation}
\label{form2}
a_2=a_1^{1+q+q^3}/a_3^{q^3+q}.
\end{equation}
By \eqref{form1} this implies
\begin{equation}
\label{form3}
a_4=a_1^{q^3+1}/a_3^{q^3}.
\end{equation}
If $\N(a_1)=\N(a_3)$, then also $\N(a_1)=\N(a_2)=\N(a_3)=\N(a_4)$. We show that in this case $Im\,(f_\varphi(x)/x)=Im\,(\Tr(x)/x)$, so we must have $\N(a_1)\neq \N(a_3)$. According to Proposition \ref{trace5} it is enough to show $(a_1/a_2)^q=a_2/a_3$ and $(a_2/a_3)^q=a_3/a_4$.
By \eqref{form1} we have $(a_1/a_2)^q=(a_3/a_4)^{q^4}$, which equals $a_2/a_3$ if and only if $(a_2/a_3)^q=a_3/a_4$, i.e. $a_3^{1+q}=a_4a_2^q$.
%so it is enough to show
Taking into account \eqref{form2} and \eqref{form3}, this equality follows from
%\[a_3^{1+q}=a_4a_2^q=(a_1^{q^3+1}/a_3^{q^3})(a_1^{q+q^2+q^4}/a_3^{q^4+q^2})=a_1^{1+q+q^2+q^3+q^4}/a_3^{q^2+q^3+q^4},\]
$\N(a_1)=\N(a_3)$.
%On the other hand, if the conditions of 2. hold, then let $\alpha_j=a_j/a_3$ for $j=0,1,2,3,4$. Then $\alpha_3=1$, $\N(\alpha_1)\neq 1$, $\alpha_2=\alpha_1^{1+q+q^3}$ and $\alpha_4=\alpha_1^{1+q^3}$. Let
%\[
%\begin{pmatrix}
%a&b\\
%c&d
%\end{pmatrix}
%=
%\begin{pmatrix}
%\alpha_1^{1+q+q^3+q^4}&1\\
%1&\alpha_1^{q^2}
%\end{pmatrix}
%\begin{pmatrix}
%1&0\\
%-\alpha_0 & 1/a_3
%\end{pmatrix},
%\]
%then
%\[
%\begin{pmatrix}
%a&b\\
%c&d
%\end{pmatrix}
%\begin{pmatrix}
%x\\
%f(x)
%\end{pmatrix}
%=\]
%\[\begin{pmatrix}
%\alpha_1^{1+q+q^3+q^4}&1\\
%1&\alpha_1^{q^2}
%\end{pmatrix}
%\begin{pmatrix}
%x\\
%\alpha_1x^q+\alpha_2x^{q^2}+\alpha_3x^{q^3}+\alpha_4x^{q^4}
%\end{pmatrix}
%=
%\begin{pmatrix}
%y\\
%y^{q^2}
%\end{pmatrix},
%\]
%with $y=x\alpha_1^{1+q+q^3+q^4}+\alpha_1x^q+\alpha_1^{q^3+q+1}x^{q^2}+x^{q^3}+\alpha_1^{q^3+1}x^{q^4}$.
\end{proof}

\section{\texorpdfstring{Proof of the main theorem}{Proof of the main theorem}}
\label{section-proof}

In this section we prove Theorem \ref{main}. In order to do this, we use the following two results and the technique developed in \cite{CSMP2016}.

\begin{lemma}[{\cite[Lemma 3.4]{CSMP2016}}]
\label{lmain}
Let $f$ and $g$ be two linearized polynomials over $\F_{q^n}$. If $Im\,(f(x)/x)=Im \,(g(x)/x)$, then for each positive integer $d$ the following holds
\[\sum_{x\in \F_{q^n}^*} \left(\frac{f(x)}{x}\right)^d = \sum_{x\in \F_{q^n}^*} \left(\frac{g(x)}{x}\right)^d.\]
\end{lemma}

\begin{lemma}[see for example {\cite[Lemma 3.5]{CSMP2016}}]
\label{lfolk}
For any prime power $q$ and integer $d$ we have $\sum_{x\in \F_{q}^*} x^d=-1$ if $q-1 \mid d$ and
$\sum_{x\in \F_{q}^*} x^d=0$ otherwise.
\end{lemma}

\begin{proposition}
\label{prop}
Let $f(x)=\sum_{i=0}^4 a_i x^{q^i}$ and $g(x)=\sum_{i=0}^4 b_i x^{q^i}$ be two $q$-polynomials over $\F_{q^5}$ such that $Im\,(f(x)/x)=Im \,(g(x)/x)$.
Then the following relations hold between the coefficients of $f$ and $g$:
\begin{equation}
\label{e0}
a_0=b_0,
\end{equation}
\begin{equation}
\label{e1}
a_1a_4^q=b_1 b_4^q,
\end{equation}
\begin{equation}
\label{e2}
a_2a_3^{q^2}=b_2 b_3^{q^2},
\end{equation}
\begin{equation}
\label{e3}
a_1^{q+1}a_3^{q^2}+a_2a_4^{q+q^2}=b_1^{q+1}b_3^{q^2}+b_2b_4^{q+q^2},
\end{equation}
\begin{equation}
\label{e4}
a_1a_2^{q+q^3}+a_3^{1+q^3}a_4^q=b_1b_2^{q+q^3}+b_3^{1+q^3}b_4^q,
\end{equation}
\begin{equation}
\label{e5}
a_1^{1+q+q^2}a_2^{q^3}+a_2^{1+q}a_3^{q^2+q^3}+a_1^qa_3^{1+q^2+q^3}+a_1^{q^2}a_2a_3^{q^3}a_4^q+a_2^{1+q+q^3}a_4^{q^2}+
\end{equation}
\[a_1^qa_2^{q^3}a_3a_4^{q^2}+a_1a_2^qa_3^{q^2}a_4^{q^3}+a_1^{1+q^2}a_4^{q+q^3}+a_3a_4^{q+q^2+q^3}=\]
\[b_1^{1+q+q^2}b_2^{q^3}+b_2^{1+q}b_3^{q^2+q^3}+b_1^qb_3^{1+q^2+q^3}+b_1^{q^2}b_2b_3^{q^3}b_4^q+b_2^{1+q+q^3}b_4^{q^2}+\]
\[b_1^qb_2^{q^3}b_3b_4^{q^2}+b_1b_2^qb_3^{q^2}b_4^{q^3}+b_1^{1+q^2}b_4^{q+q^3}+b_3b_4^{q+q^2+q^3},\]
\begin{equation}
\label{e6}
\N(a_1)+\N(a_2)+\N(a_3)+\N(a_4)+\Tr(a_1^qa_2^{q^2+q^3+q^4}a_3+a_1^{q+q^3}a_2^{q^4}a_3^{1+q^2}+
\end{equation}
\[a_1^{q+q^2}a_2^{q^3+q^4}a_4+a_1^{q+q^2+q^4}a_3^{q^3}a_4+a_2^qa_3^{q^2+q^3+q^4}a_4+a_1^{q^2}a_3^{q^3+q^4}a_4^{1+q}+\]
\[a_2^{q+q^3}a_3^{q^4}a_4^{1+q^2}+a_1^{q^2}a_2^{q^4}a_4^{1+q+q^3})=\]
\[\N(b_1)+\N(b_2)+\N(b_3)+\N(b_4)+\Tr(b_1^qb_2^{q^2+q^3+q^4}b_3+b_1^{q+q^3}b_2^{q^4}b_3^{1+q^2}+\]
\[b_1^{q+q^2}b_2^{q^3+q^4}b_4+b_1^{q+q^2+q^4}b_3^{q^3}b_4+b_2^qb_3^{q^2+q^3+q^4}b_4+b_1^{q^2}b_3^{q^3+q^4}b_4^{1+q}+\]
\[b_2^{q+q^3}b_3^{q^4}b_4^{1+q^2}+b_1^{q^2}b_2^{q^4}b_4^{1+q+q^3}).\]
\end{proposition}
\begin{proof}

Equations \eqref{e0}--\eqref{e4} follow from \cite[Lemma 3.6]{CSMP2016}.
To prove \eqref{e5} we will use Lemma \ref{lfolk} together with Lemma \ref{lmain} with $d=q^3+q^2+q+1$. This gives us
\[\sum_{1\leq i,j,m,n \leq 4} a_ia_j^qa_m^{q^2}a_n^{q^3} \sum_{x\in \F_{q^5}^*} x^{q^i-1+q^{j+1}-q+q^{m+2}-q^2+q^{n+3}-q^3}=\]
\[\sum_{1\leq i,j,m \leq 4} b_ib_j^qb_m^{q^2}b_n^{q^3} \sum_{x\in \F_{q^5}^*} x^{q^i-1+q^{j+1}-q+q^{m+2}-q^2+q^{n+3}-q^3}.\]
We have $\sum_{x\in \F_{q^n}^*} x^{q^i-1+q^{j+1}-q+q^{m+2}-q^2+q^{n+3}-q^3}=-1$ if and only if
\begin{equation}
\label{eki}
q^i+q^{j+1}+q^{m+2}+q^{n+3} \equiv 1+q+q^2+q^3 \pmod {q^5-1},
\end{equation}
and zero otherwise. Suppose that the former case holds. The right-hand side of \eqref{eki} is smaller than the left-hand side, thus
\[q^i+q^{j+1}+q^{m+2}+q^{n+3} = 1+q+q^2+q^3 + k(q^5-1),\]
for some positive integer $k$.
We have $q^i+q^{j+1}+q^{m+2}+q^{n+3}\leq q^4+q^5+q^6+q^7< 1+q+q^2+q^3+(q^2+q+2)(q^5-1)$ and hence $k\leq q^2+q+1$.
If $i=1$, then $q^2 \mid 1-k$ and hence $k=1$, $j=m=1$ and $n=2$, or $k=q^2+1$, $n=4$ and either $j=2$ and $m=3$, or $j=4$ and $m=1$.
If $i>1$, then $q^2$ divides $q+1-k$ and hence $k=q+1$, or $k=q^2+q+1$.
In the former case $i=j=n=2$ and $m=4$, or $i=j=2$ and $n=m=3$, or $i=3$, $j=1$, $m=4$ and $n=2$, or $i=3$, $j=1$ and $m=n=3$,
or $m=1$, $i=2$, $j=4$ and $n=3$. In the latter case $i=3$ and $n=m=j=4$.
Then \eqref{e5} follows.

To prove \eqref{e6} we follow the previous approach with $d=q^4+q^3+q^2+q+1$.
We obtain
\[\sum a_ia_j^qa_m^{q^2}a_n^{q^3}a_r^{q^4}=\sum b_ib_j^qb_m^{q^2}b_n^{q^3}b_r^{q^4},\]
where the summation is on the quintuples $(i,j,m,n,r)$ with elements taken from $\{1,2,3,4\}$ such that  $L_{i,j,m,n,r}:=(q^i-1)+(q^{j+1}-q)+(q^{m+2}-q^2)+(q^{n+3}-q^3)+(q^{r+4}-q^4)$ is divisible by $q^5-1$.
Then
\[L_{i,j,m,n,r} \equiv K_{i,j',m',n',r'} \pmod {q^5-1},\]
where
\[K_{i,j',m',n',r'}=(q^i-1)+(q^{j'}-q)+(q^{m'}-q^2)+(q^{n'}-q^3)+(q^{r'}-q^4),\]
such that
\begin{equation}
\label{trex}
j'\in \{0,2,3,4\}, \quad m'\in \{0,1,3,4\}, \quad n'\in \{0,1,2,4\}, \quad r'\in \{0,1,2,3\}
\end{equation}
with $j'\equiv j+1,\, m'\equiv m+2,\, n'\equiv n+3,\, r'\equiv r+4 \pmod 5$.

For $q=2$ and $q=3$ we can determine by computer those quintuples $(i,j',m',n',r')$ for which $K_{i,j',m',n',r'}$ is divisible by $q^5-1$ and hence \eqref{e6} follows. So we may assume $q>3$. Then
\[3-q^2-q^3-q^4=(q-1)+(1-q)+(1-q^2)+(1-q^3)+(1-q^4) \leq \]
\[K_{i,j',m',n',r'}\leq \]
\[(q^4-1)+(q^4-q)+(q^4-q^2)+(q^4-q^3)+(q^3-q^4)=3q^4-1-q-q^2,\]
and hence $L_{i,j,m,n,r}$ is divisible by $q^5-1$ if and only if $K_{i,j',m',n',r'}=0$.
It follows that
\[q^i+q^{j'}+q^{m'}+q^{n'}+q^{r'}=1+q+q^2+q^3+q^4.\]
So $\sum_{h=0}^4 c_h q^h=1+q+q^2+q^3+q^4$ for some $0\leq c_h \leq 4$ with $\sum_{h=0}^4 c_h=5$. For $q>3$ this happens only if $c_h=1$ for $h=0,1,2,3,4$ thus
we have to find those quintuples $(i,j',m',n',r')$ for which $i\in \{1,2,3,4\}$, $\{i,j',m',n',r'\}=\{0,1,2,3,4\}$ and \eqref{trex} are satisfied. This can  be done by computer and the 44 solutions yield \eqref{e6}.
\end{proof}

%\begin{theorem}
%Let $f(x)$ and $g(x)$ be two $q$-polynomials over $\F_{q^5}$ such that $L_f=L_g$. Then
%either $L_f=L_g$ is a linear set of pseudoregulus type, or there exists $\lambda\in \F_{q^5}^*$ for which
%\[g(x)=f(\lambda x)/\lambda \quad \text{or} \quad g(x)=\hat{f}(\lambda x)/\lambda.\]
%\end{theorem}

\subsection*{Proof of Theorem \ref{main}}

%\begin{proof}
Since $f$ has maximum field of linearity $\F_q$, we cannot have $a_1=a_2=a_3=a_4=0$. If three of $\{a_1,a_2,a_3,a_4\}$ are zeros, then $f(x)=a_0x+a_ix^{q^i}$, for some $i\in\{1,2,3,4\}$. Hence with $\varphi$ represented by $\begin{pmatrix}
	1 & 0 \\
	-a_0/a_i & 1/a_i \\
\end{pmatrix}$ we have $f_{\varphi}(x)=x^{q^i}$. Then Proposition \ref{imtrans} and Corollary \ref{norm} give $g_\varphi(x)=\beta x^{q^j}$ where $\N(\beta)=1$ and   $j\in\{1,2,3,4\}$.
Now, we distinguish three main cases according to the number of zeros among $\{a_1,a_2,a_3,a_4\}$.

%If $L_f=L_g$ is a point (which happens when $a_0=b_0$ and $a_i=b_i=0$ for $i\in\{1,2,3,4\}$), then the result is trivial. If three of the coefficients in $\{a_1,a_2,a_3,a_4\}$ are zeros, then by definition $L_f$ is a linear set of pseudoregulus type. We distinguish three main cases according to the number of zeros among $\{a_1,a_2,a_3,a_4\}$.

\subsection*{Two zeros among $\{a_1, a_2, a_3, a_4\}$}

Applying Proposition \ref{prop} we obtain $a_0=b_0$.
The two non-zero coefficients can be chosen in six different ways, however the cases $a_1a_2\neq 0$ and $a_1a_3\neq 0$ correspond to $a_3a_4\neq 0$ and $a_2a_4\neq 0$, respectively,
since $Im\,(f(x)/x)=Im \,(\hat f(x)/x)$. Thus, after possibly interchanging $f$ with $\hat f$, we may consider only four cases.
\medskip

First let $f(x)=a_0x+a_1x^q+a_4x^{q^4}$, $a_1a_4\neq 0$.
Applying Proposition \ref{prop} we obtain
\begin{equation}
\label{f1}
a_1a_4^q=b_1 b_4^q,
\end{equation}
\[0=b_2 b_3^{q^2}.\]
Since $b_1b_4\neq 0$, from \eqref{e3} we get $b_2=b_3=0$ and hence \eqref{e6} gives
%\[0=b_2=b_1^{q+1}b_3^{q^2}=b_3^{1+q^3}b_4^q,\]
%or
%\[0=b_3=b_2b_4^{q+q^2}=b_1b_2^{q+q^3}.\]
%If $b_2=0$ and $b_1=0$ then $b_3$ or $b_4$ is also zero and hence $L_f$ is of pseudoregulus type (cf. Introduction), a contradiction. It follows $b_2=0$ implies $b_3=0$. If $b_3=0$, then in the same way $b_2=0$ follows.
\[ \N(a_1)+\N(a_4)=\N(b_1)+\N(b_4).\]
Also, from \eqref{f1} we have $\N(a_1a_4)=\N(b_1b_4)$.
It follows that either $\N(a_1)=\N(b_1)$ and $\N(a_4)=\N(b_4)$, or $\N(a_1)=\N(b_4)$ and $\N(a_4)=\N(b_1)$.
In the first case $b_1=a_1\lambda^{q-1}$ for some $\lambda\in \F_{q^5}^*$ and by \eqref{f1}
we get $g(x)=f(\lambda x)/\lambda$. In the latter case  $b_1=a_4^q\lambda^{q-1}$ for some $\lambda\in \F_{q^5}^*$ and by \eqref{f1} we get  $g(x)=\hat{f}(\lambda x)/\lambda$.
\medskip

Now consider $f(x)=a_1x^q+a_3x^{q^3}$, $a_1a_3\neq 0$.
Applying Proposition \ref{prop} and arguing as above we have either $b_2=b_4=0$ or $b_1=b_3=0$.
In the first case from \eqref{e5}
%\begin{equation}
%0=b_1 b_4^q,
%\end{equation}
%\begin{equation}
%0=b_2 b_3^{q^2},
%\end{equation}
%and hence $b_2=0$ and
%\begin{equation}
%\label{m1}
%a_1^{q+1}a_3^{q^2}=b_1^{q+1}b_3^{q^2},
%\end{equation}
%\begin{equation}
%0=b_3^{1+q^3}b_4^q,
%\end{equation}
%or $b_3=0$ and
%\begin{equation}
%\label{m2}
%a_1^{q+1}a_3^{q^2}=b_2b_4^{q+q^2},
%\end{equation}
%\begin{equation}
%0=b_1b_2^{q+q^3}.
%\end{equation}
%If $b_2=0$, then we cannot have $b_3=0$ since then one of $b_1$ and $b_4$ is also zero, but $L_f$ cannot be of pseudoregulus type (cf. Result \ref{pseudo}). It follows that $b_4=0$. If $b_3=0$, then we obtain similarly $b_2\neq 0$ and hence $b_1=0$.
%In the $b_2=b_4=0$ case from \eqref{e5} we obtain
we obtain
\[a_1^qa_3^{1+q^2+q^3}=b_1^qb_3^{1+q^2+q^3}\]
and together with \eqref{e3} this yields $\N(a_1)=\N(b_1)$ and $\N(a_3)=\N(b_3)$.
In this case $g(x)=f(\lambda x)/\lambda$ for some $\lambda\in \F_{q^5}^*$.
If $b_1=b_3=0$, then in $\hat{g}(x)$ the coefficients of $x^{q^2}$ and $x^{q^4}$ are zeros thus applying the result obtained in the former case we get $\lambda\hat{g}(x)=f(\lambda x)$ and hence after substituting $y=\lambda x$ and taking the adjoints of both sides we obtain $g(y)=\hat{f}(\mu y)/\mu$, where $\mu=\lambda^{-1}$.

%from \eqref{e5} we obtain
%\begin{equation}
%a_1^qa_3^{1+q^2+q^3}=b_2^{1+q+q^3}b_4^{q^2}.
%\end{equation}
%Together with \eqref{m2} this yields $\N(a_3)=\N(b_2)$ and hence $\N(a_1)=\N(b_4)$. In this case $g(x)=\hat{f}(\lambda x)/\lambda$.
\medskip

The cases $f(x)=a_1x^q+a_2x^{q^2}$ and $f(x)=a_2x^{q^2}+a_3x^{q^3}$ can be handled in a similar way, applying Equations \eqref{e1}--\eqref{e6} of Proposition \ref{prop}.

\subsection*{One zero among $\{a_1, a_2, a_3, a_4\}$}

Since $Im\,(f(x)/x)=Im \,(\hat f(x)/x)$, we may assume $a_1=0$ or $a_2=0$.
First suppose $a_1=0$. Then by \eqref{e1} either $b_1=0$ or $b_4=0$. In the former case
putting together Equations \eqref{e2}, \eqref{e3}, \eqref{e4} we get $\N(a_i)=\N(b_i)$ for $i\in \{2,3,4\}$ and hence there exists $\lambda\in \F_{q^5}^*$ such that $g(x)=f(\lambda x)/\lambda$.
If $a_1=b_4=0$, then in $\hat{g}(x)$ the coefficients of $x^q$ is zero thus applying the previous result we get $g(x)=\hat{f}(\mu x)/\mu$, where $\mu=\lambda^{-1}$.

Now suppose $a_2=0$. Then by \eqref{e2} either $b_2=0$ or $b_3=0$. Using the same approach but applying \eqref{e1}, \eqref{e3} and \eqref{e4} we obtain $g(x)=f(\lambda x)/\lambda$ or $g(x)=\hat{f}(\lambda x)/\lambda$.

\subsection*{Case $a_1a_2a_3a_4\neq 0$}

We will apply \eqref{e0}-\eqref{e5} of Proposition \ref{prop}.
Note that Equations \eqref{e1} and \eqref{e2} yield $a_1a_2a_3a_4\neq 0 \Leftrightarrow b_1b_2b_3b_4\neq 0$.
Multiplying \eqref{e3} by $a_2$ and applying \eqref{e2} yield
\[a_2^2a_4^{q+q^2}-a_2(b_1^{q+1}b_3^{q^2}+b_2b_4^{q+q^2})+a_1^{q+1}b_3^{q^2}b_2=0.\]
Taking \eqref{e1} into account, this is equivalent to \[(a_2a_4^{q+q^2}-b_1^{q+1}b_3^{q^2})(a_2a_4^{q+q^2}-b_2b_4^{q+q^2})=0.\]
Multiplying \eqref{e4} by $a_1$ and applying \eqref{e1} yield
\[a_1^2a_2^{q+q^3}-a_1(b_1b_2^{q+q^3}+b_3^{1+q^3}b_4^q)+a_3^{1+q^3}b_4^qb_1=0.\]
Taking \eqref{e2} into account, this is equivalent to
\[(a_1a_2^{q+q^3}-b_1b_2^{q+q^3})(a_1a_2^{q+q^3}-b_3^{1+q^3}b_4^q)=0.\]
We distinguish four cases:
\begin{enumerate}
	\item $a_2a_4^{q+q^2}=b_1^{q+1}b_3^{q^2}$ and $a_1a_2^{q+q^3}=b_1b_2^{q+q^3}$,
	\item $a_2a_4^{q+q^2}=b_1^{q+1}b_3^{q^2}$ and $a_1a_2^{q+q^3}=b_3^{1+q^3}b_4^q$,
	\item $a_2a_4^{q+q^2}=b_2b_4^{q+q^2}$ and $a_1a_2^{q+q^3}=b_1b_2^{q+q^3}$,
	\item $a_2a_4^{q+q^2}=b_2b_4^{q+q^2}$ and $a_1a_2^{q+q^3}=b_3^{1+q^3}b_4^q$.
\end{enumerate}
We show that these four cases produce the relations:
\begin{equation}
\label{1per4}
\N\left(\frac{b_1}{a_4}\right)=\frac{a_1a_2^{q+q^3}}{a_4^qa_3^{q^3+1}}=\frac{b_1b_2^{q+q^3}}{b_4^qb_3^{q^3+1}},
\end{equation}
\begin{equation}
\label{2per4}
\N\left(\frac{b_1}{a_4}\right)=1,
\end{equation}
\begin{equation}
\label{3per4}
\N\left(\frac{b_1}{a_1}\right)=1,
\end{equation}
\begin{equation}
\label{4per4}
\N\left(\frac{b_1}{a_1}\right)=\frac{a_3^{q^3+1}a_4^q}{a_1 a_2^{q+q^3}}=\frac{b_1b_2^{q+q^3}}{b_3^{q^3+1}b_4^q},
\end{equation}
respectively.
To see \eqref{1per4} observe that from $a_2a_4^{q+q^2}=b_1^{q+1}b_3^{q^2}$ and \eqref{e1} we get
\begin{equation}
\label{new}
\N\left(\frac{b_1}{a_4}\right)=\left(\frac{b_1^{q+1}}{a_4^{q+q^2}}\right)^{q^2+1}\frac{b_1^{q^4}}{a_4}=
\left(\frac{a_2^{q^2+1}}{b_3^{q^2+q^4}}\right)\frac{a_1^{q^4}}{b_4}=\frac{a_1a_2^{q+q^3}}{b_4^qb_3^{q^3+1}},
\end{equation}
and hence by $a_1a_2^{q+q^3}=b_1b_2^{q+q^3}$ and \eqref{e4} we get \eqref{1per4}.
Equation \eqref{2per4} immediately follows from \eqref{new} taking $a_1a_2^{q+q^3}=b_3^{1+q^3}b_4^q$ into account.
Similarly, using \eqref{e1} and \eqref{e4} we get \eqref{3per4} and \eqref{4per4}.

In Case 3 by \eqref{3per4} we get $b_1=a_1\lambda^{q-1}$ for some $\lambda\in \F_{q^5}^*$ and by \eqref{e1} and \eqref{e2} we have $g(x)=f(\lambda x)/\lambda$. Analogously, in Case 2 $g(x)=\hat{f}(\lambda x)/\lambda$.
%Note that in the first case \eqref{e3} and \eqref{e4} yield
%\begin{equation}
%\label{casoegy3}
%a_1^{q+1}a_3^{q^2}=b_2b_4^{q+q^2}
%\end{equation}
%and
%\begin{equation}
%\label{casoegy4}
%a_3^{1+q^3}a_4^q=b_3^{1+q^3}b_4^q.
%\end{equation}
Note that Case 4 is just Case 3 after replacing $g$ by $\hat{g}$ since $Im\,(g(x)/x)=Im \,(\hat g(x)/x)$.
This allows us to restrict ourself to Case 1. It will be useful to express $a_1$, $a_2$, $a_3$ as follows:
\begin{equation}
\label{ais}
a_1=\frac{b_1 b_4^q}{a_4^q}, \quad a_2=\frac{b_1^{q+1}b_3^{q^2}}{a_4^{q+q^2}}, \quad a_3=\frac{b_2^{q^3}b_4^{1+q^4}}{a_1^{q^3+q^4}}.
\end{equation}
%Applying the above form of $a_3$ and \eqref{casoegy4} it follows that
%\[\left(\frac{b_2^{q^3}a_4^{q^4+1}}{b_1^{q^4+q^3}}\right)^{q^3+1}=a_3^{q^3+1}=\frac{b_4^qb_3^{q^3+1}}{a_4^q}.\]
%After rearranging, the equality of the expressions on the left and the right yields \eqref{1per4}.

%The first and the fourth cases can occur only if $a_1a_2^{q+q^3}/a_4^qa_3^{q^3+1}\in \F_{q}$, or equivalently
%$b_1b_2^{q+q^3}/b_3^{q^3+1}b_4^q\in \F_q$.
We are going to simplify \eqref{e5}.
Using Equations \eqref{ais} and \eqref{e1} it is easy to see that $a_2^{1+q}a_3^{q^2+q^3}=b_2^{1+q}b_3^{q^2+q^3}$, $a_1^{1+q^2}a_4^{q+q^3}=b_1^{1+q^2}b_4^{q+q^3}$, $a_1^{q^2}a_2a_3^{q^3}a_4^q=b_1b_2^qb_3^{q^2}b_4^{q^3}$, $a_1^qa_2^{q^3}a_3a_4^{q^2}=b_1^qb_2^{q^3}b_3b_4^{q^2}$, $a_1a_2^qa_3^{q^2}a_4^{q^3}=b_1^{q^2}b_2b_3^{q^3}b_4^q$ and hence
\begin{equation}
\label{eq5b}
a_1^{1+q+q^2}a_2^{q^3}+a_1^qa_3^{1+q^2+q^3}+a_2^{1+q+q^3}a_4^{q^2}+a_3a_4^{q+q^2+q^3}=
\end{equation}
\[b_1^{1+q+q^2}b_2^{q^3}+b_1^qb_3^{1+q^2+q^3}+b_2^{1+q+q^3}b_4^{q^2}+b_3b_4^{q+q^2+q^3}.\]
The following equations can be proved applying \eqref{e1}, \eqref{e2} and \eqref{ais}:
\begin{equation}
\label{ea}
\N\left(\frac{b_1}{a_4}\right)b_3b_4^{q+q^2+q^3}=a_2^{q^3}a_1^{1+q+q^2},
\end{equation}
\begin{equation}
\label{eb}
\N\left(\frac{a_4}{b_1}\right)b_4^{q^2}b_2^{1+q+q^3}=a_1^qa_3^{1+q^2+q^3},
\end{equation}
\begin{equation}
\label{ec}
\N\left(\frac{b_1}{a_4}\right)b_1^qb_3^{1+q^2+q^3}=a_2^{1+q+q^3}a_4^{q^2},
\end{equation}
\begin{equation}
\label{ed}
\N\left(\frac{a_4}{b_1}\right)b_2^{q^3}b_1^{1+q+q^2}=a_3a_4^{q+q^2+q^3}.
\end{equation}

Then \eqref{eq5b} can be written as
\[(\N(b_1/a_4)-1)(b_3b_4^{q+q^2+q^3}+b_1^qb_3^{1+q^2+q^3})=\frac{\N(b_1/a_4)-1}{\N(b_1/a_4)}(b_4^{q^2}b_2^{1+q+q^3}+b_2^{q^3}b_1^{1+q+q^2}).\]
If $\N(b_1/a_4)=1$, then we are in Case 2 and hence the assertion follows. Otherwise dividing by $\N(b_1/a_4)-1$ and substituting $\N(b_1/a_4)=b_1b_2^{q+q^3}/b_4^qb_3^{q^3+1}$ we obtain
\[ b_1b_2^{q+q^3}(b_3b_4^{q+q^2+q^3}+b_1^qb_3^{1+q^2+q^3})=b_4^qb_3^{q^3+1}(b_4^{q^2}b_2^{1+q+q^3}+b_2^{q^3}b_1^{1+q+q^2}).\]
Substituting $\N(b_1/a_4) b_4^qb_3^{q^3+1}/b_2^{q+q^3}$ for $b_1$ and using the fact that $\N(b_1/a_4)\in \F_q$ we obtain
\[ \left(1-\N\left(\frac{b_1}{a_4}\right)^2\N\left(\frac{b_3}{b_2}\right)\right)\left(\N\left(\frac{b_1}{a_4}\right)b_4^{q+q^3}b_3-b_2^{1+q+q^3}\right)=0.\]
This gives us two possibilities:
\begin{equation}
\label{sec}
\N\left(\frac{b_1}{a_4}\right)b_4^{q+q^3}b_3=b_2^{1+q+q^3},
\end{equation}
or
\begin{equation}
\label{sec2}
\N\left(\frac{b_2}{b_3}\right)=\N\left(\frac{b_1}{a_4}\right)^2.
\end{equation}
First consider the case when \eqref{sec2} holds. We show $\N(a_1)=\N(b_1)$, that is, \eqref{3per4}.
We have $a_2a_4^{q+q^2}=b_1^{q+1}b_3^{q^2}$ from \eqref{ais} and hence $\N(a_2)\N(a_4)^2=\N(b_1)^2\N(b_3)$. It follows that
\[\N\left(\frac{b_1}{a_4}\right)^2=\N\left(\frac{a_2}{b_3}\right).\]
Combining this with \eqref{sec2} we obtain $\N(b_2)=\N(a_2)$. Then $\N(b_1)=\N(a_1)$ follows from $a_1a_2^{q+q^3}=b_1b_2^{q+q^3}$ since we are in Case 1.

From now on we can suppose that \eqref{sec} holds.
%Applying \eqref{ais} and then \eqref{1per4} yields
%\[A_1=\N(b_1/a_4)b_2^{q^3}b_3^{1+q+q^4}b_4^{q^2}=\frac{b_1^qb_2^{q^2+q^4}}{b_4^{q^2}b_3^{q^4+q}}b_2^{q^3}b_3^{1+q+q^4}b_4^{q^2}=B_1.\]
%Applying \eqref{ais} we obtain
%\[A_2=B_2^{q^2}.\]
%Applying \eqref{ais} and then \eqref{1per4} yields
%\[A_3=\]
Then \eqref{1per4} yields
\begin{equation}
\label{elso}
\left(\frac{b_1}{b_2}\right)^{q^2}=\frac{b_3}{b_4}.
\end{equation}
Multiplying both sides of \eqref{sec} by $b_4^{q^2}$ and applying \eqref{ea} gives
\begin{equation}
\label{nuovo}
a_2^{q^3}a_1^{1+q+q^2}=b_2^{1+q+q^3}b_4^{q^2}.
\end{equation}
Then multiplying \eqref{ea} by \eqref{eb} and taking \eqref{nuovo} into account we obtain
%\[\N\left(\frac{a_4}{b_1}\right)\N\left(\frac{b_1}{a_4}\right)b_3b_4^{q+q^2+q^3}=\N\left(\frac{a_4}{b_1}\right)b_4^{q^2}b_2^{1+q+q^3}=a_1^qa_3^{1+q^2+q^3},\]
%and hence
\[a_1^qa_3^{1+q^2+q^3}=b_3b_4^{q+q^2+q^3}.\]
Multiplying \eqref{ec} and \eqref{ed} yield
\[(b_1^qb_3^{1+q^2+q^3})(b_2^{q^3}b_1^{1+q+q^2})=(a_2^{1+q+q^3}a_4^{q^2})(a_3a_4^{q+q^2+q^3}).\]
On the other hand, from \eqref{eq5b} it follows that
\[b_1^qb_3^{1+q^2+q^3}+b_2^{q^3}b_1^{1+q+q^2}=a_2^{1+q+q^3}a_4^{q^2}+a_3a_4^{q+q^2+q^3}.\]
Hence $b_1^qb_3^{1+q^2+q^3}=a_2^{1+q+q^3}a_4^{q^2}$ and $b_2^{q^3}b_1^{1+q+q^2}=a_3a_4^{q+q^2+q^3}$, or
$b_1^qb_3^{1+q^2+q^3}=a_3a_4^{q+q^2+q^3}$ and $b_2^{q^3}b_1^{1+q+q^2}=a_2^{1+q+q^3}a_4^{q^2}$.
In the former case \eqref{ec} yields $\N(b_1/a_4)=1$, which is \eqref{2per4}.
In the latter case \eqref{1per4} and \eqref{ed} gives
\[\frac{b_4^qb_3^{q^3+1}}{b_1b_2^{q+q^3}}b_2^{q^3}b_1^{1+q+q^2}=\N(a_4/b_1)b_2^{q^3}b_1^{1+q+q^2}=b_1^qb_3^{1+q^2+q^3},\]
and hence
\begin{equation}
\label{masodik}
\frac{b_4}{b_2}=\left(\frac{b_3}{b_1}\right)^q.
\end{equation}
Equation \eqref{elso} is equivalent to
\[b_4b_1^{q^2}=b_3b_2^{q^2},\]
while \eqref{masodik} is equivalent to
\[b_4b_1^q=b_3^qb_2.\]
Dividing these two equations by each other yield
\[b_2^{q^2-1}=b_3^{q-1}b_1^{q^2-q}.\]
It follows that there exists $\lambda\in \F_q^*$ such that
\begin{equation}
\label{lambda}
b_2^{q+1}=\lambda b_3 b_1^q,
\end{equation}
thus
\[b_3=b_2^{q+1}/(b_1^q\lambda)\]
and
\[b_4=b_2^{1+q+q^2}/(b_1^{q+q^2}\lambda).\]
Then \eqref{1per4} can be written as
\[\N\left(\frac{b_1}{a_4}\right)=\frac{b_1b_2^{q+q^3}}{b_4^qb_3^{q^3+1}}=\N\left(\frac{b_1}{b_2}\right)\lambda^3,\]
and hence
\[\N\left(\frac{b_2}{a_4}\right)=\lambda^3.\]
Then, using the previous expressions for $b_3$ and $b_4$ and taking  \eqref{e1} and  \eqref{e2} into account,  we can express $\N(a_i)$ for $i=1,2,3,4$ as follows
\begin{equation}
\label{bp1}
\N(a_1)=\N(b_2)^2/(\N(b_1)\lambda^2),
\end{equation}
\begin{equation}
\label{bp2}
\N(a_2)=\N(b_1)\lambda,
\end{equation}
\[\N(a_3)=\N(b_2)^3/(\N(b_1)^2\lambda^6),\]
\[\N(a_4)=\N(b_2)/\lambda^3.\]

Before we go further, we simplify \eqref{e6} and prove
\begin{equation}
\label{sumofnorms}
\N(a_1)+\N(a_2)+\N(a_3)+\N(a_4)=\N(b_1)+\N(b_2)+\N(b_3)+\N(b_4).
\end{equation}
It is enough to show
\[\Tr(\overbrace{\hbox{$a_1^qa_2^{q^2+q^3+q^4}a_3$}}^{A_1}+\overbrace{\hbox{$a_1^{q+q^3}a_2^{q^4}a_3^{1+q^2}$}}^{A_2}+
\overbrace{\hbox{$a_1^{q+q^2}a_2^{q^3+q^4}a_4$}}^{A_3}+\overbrace{\hbox{$a_1^{q+q^2+q^4}a_3^{q^3}a_4$}}^{A_4}+\]
\[\overbrace{\hbox{$a_2^qa_3^{q^2+q^3+q^4}a_4$}}^{A_5}+\overbrace{\hbox{$a_1^{q^2}a_3^{q^3+q^4}a_4^{1+q}$}}^{A_6}+
\overbrace{\hbox{$a_2^{q+q^3}a_3^{q^4}a_4^{1+q^2}$}}^{A_7}+\overbrace{\hbox{$a_1^{q^2}a_2^{q^4}a_4^{1+q+q^3}$}}^{A_8})=\]
\[\Tr(\overbrace{\hbox{$b_1^qb_2^{q^2+q^3+q^4}b_3$}}^{B_1}+\overbrace{\hbox{$b_1^{q+q^3}b_2^{q^4}b_3^{1+q^2}$}}^{B_7}+
\overbrace{\hbox{$b_1^{q+q^2}b_2^{q^3+q^4}b_4$}}^{B_3}+\overbrace{\hbox{$b_1^{q+q^2+q^4}b_3^{q^3}b_4$}}^{B_8}+\]
\[\overbrace{\hbox{$b_2^qb_3^{q^2+q^3+q^4}b_4$}}^{B_5}+\overbrace{\hbox{$b_1^{q^2}b_3^{q^3+q^4}b_4^{1+q}$}}^{B_6}+
\overbrace{\hbox{$b_2^{q+q^3}b_3^{q^4}b_4^{1+q^2}$}}^{B_2}+\overbrace{\hbox{$b_1^{q^2}b_2^{q^4}b_4^{1+q+q^3}$}}^{B_4}),\]
which can be done by proving $\Tr(A_i)=\Tr(B_i)$ for $i=1,2,\ldots,8$.
Expressing $a_3$ with $a_4$ in \eqref{ais} gives $a_3=b_2^{q^3}a_4^{q^4+1}/b_1^{q^3+q^4}$.
Then $a_1,a_2,a_3$ can be eliminated in all of the $A_i$, $i\in \{1,2\ldots,8\}$.
It turns out that this procedure eliminates also $a_4$ when $i\in\{2,4,7,8\}$ and we obtain
$A_2=B_2^{q^2}$, $A_4=B_4^{q^2}$, $A_7=B_7^{q^3}$ and $A_8=B_8^{q^2}$.
In each of the other cases what remains is $\N(a_4)$ times an expression in $b_1,b_2,b_3,b_4$.
Then by using \eqref{1per4} we can also eliminate $\N(a_4)$ and hence $A_i$ can be expressed in terms of $b_1,b_2,b_3,b_4$. This gives $A_1=B_1$ and $A_5=B_5$. Applying also $\eqref{elso}$ and $\eqref{masodik}$ we obtain $A_3=B_3^{q^2}$ and $A_6=B_6$.
%(One has to use \eqref{ais} and sometimes also \eqref{elso} and \eqref{masodik}.)

Let $x=\N(b_2/b_1)$, then \eqref{sumofnorms} gives us the following equation
\[x^2\lambda^4+\lambda^7+x^3+x\lambda^3=\lambda^6+x\lambda^6+x^2\lambda+\lambda x^3.\]
After rearranging we get:
\[(1-\lambda)(x-\lambda)(x-\lambda^2)(x-\lambda^3)=0.\]
First suppose $\lambda\neq 1$, then we have three possibilities:
\[x=\lambda,\]
in which case $N(b_1)=N(a_1)$ follows from \eqref{bp1}, which is Case 3;
\[x=\lambda^3,\]
in which case $N(a_4)=N(b_1)$ follows from \eqref{bp2}, which is Case 2;
\[x=\lambda^2,\]
in which case we show %that $g$ is equivalent either to $x^{q^i}$ or to the trace function, and we have the assertion from Proposition \ref{imtrans} and Corollary \ref{norm} or from  \cite[Theorem 3.7]{CSMP2016}, respectively.
that there exists $\varphi\in\Gamma\mathrm{L}(2,q^5)$ such that either $Im\,(g_\varphi(x)/x)=Im\,(x^{q-1})$ or $Im\,(g_\varphi(x)/x)=Im\,(\Tr(x)/x)$. In the former case by Proposition \ref{imtrans} and Corollary \ref{norm} we get $f_\varphi(x)=\alpha x^{q^i}$ and $g_\varphi(x)=\beta x^{q^j}$ for some $i,j\in\{1,2,3,4\}$, with $\N(\alpha)=\N(\beta)=1$. In the latter case, by Theorem \ref{trace} and by Propositions \ref{imtrans} and \ref{lambda0}, there exists $\lambda\in\F_{q^5}^*$ such that $g(x)=f(\lambda x)/\lambda$.

According to Proposition \ref{pseudoalg} part 2, it is enough to show
\[(b_4/b_1)^{q^2}=b_1/b_3, \quad (b_1/b_2)^{q^2}=b_3/b_4.\]
The second equation is just \eqref{elso}, thus it is enough to prove the first one.
First we show
\begin{equation}
\label{meta}
b_2b_3^{q+q^3}=b_1^{1+q+q^3}.
\end{equation}
From \eqref{lambda} we have
\[\N\left(\frac{b_2}{b_1}\right)=\lambda^2=\left(\frac{b_2^{q+1}}{b_3b_1^q}\right)^2,\]
and hence after rearranging
\[\frac{b_2^{q^2+q^3+q^4}b_3}{b_1^{1+q^2+q^3+q^4}}=\frac{b_2^{q+1}}{b_3b_1^q}.\]
On the right-hand side we have $\lambda$, which is in $\F_q$, thus, after taking $q$-th powers on the left and $q^3$-th powers on the right, the following also holds
\[\frac{b_2^{q^3+q^4+1}b_3^q}{b_1^{q+q^3+q^4+1}}=\frac{b_2^{q^3+q^4}}{b_3^{q^3}b_1^{q^4}}.\]
After rearranging we obtain \eqref{meta}.
Now we show that $(b_4/b_1)^{q^2}=b_1/b_3$ is equivalent to \eqref{meta}.
Expressing $b_4$ from \eqref{elso} we get
\[(b_4/b_1)^{q^2}=b_1/b_3 \Leftrightarrow b_3^{1+q^2}b_2^{q^4}=b_1^{1+q^2+q^4},\]
where the equation on the right-hand side is just the $q^4$-th power of \eqref{meta}.

Now consider the case when $\lambda=1$. Then $b_3=b_2^{q+1}/b_1^q$, $b_4=b_2^{1+q+q^2}/b_1^{q+q^2}$ and it follows from Proposition \ref{pseudoalg}
that there exists $\varphi\in\Gamma\mathrm{L}(2,q^5)$ such that either $Im\,(g_\varphi(x)/x)=Im\,(x^{q-1})$ or $Im\,(g_\varphi(x)/x)=Im\,(\Tr(x)/x)$. As above, the assertion follows either from Proposition \ref{imtrans} and Corollary \ref{norm} or from Theorem \ref{trace} and by Propositions \ref{imtrans} and \ref{lambda0}.

% $g$ is equivalent either to $x^{q^i}$ or to the trace function. The assertion follows from Proposition \ref{imtrans} and Corollary \ref{norm} or from Theorem \ref{trace}, respectively.

This finishes the proof when $\prod_{i=1}^4 a_ib_i \neq 0$.

%\end{proof}
\qed

\section{\texorpdfstring{New maximum scattered linear sets of $\PG(1,q^5)$}{New maximum scattered linear sets of PG(1,q5)}}
\label{geom}

A point set $L$ of a line  $\Lambda=\PG(W,\F_{q^n})\allowbreak=\PG(1,q^n)$ is said to be an \emph{$\F_q$-linear set} of $\Lambda$ of rank $n$ if it is
defined by the non-zero vectors of an $n$-dimensional $\F_q$-vector subspace $U$ of the two-dimensional $\F_{q^n}$-vector space $W$, i.e.
\[L=L_U:=\{\la {\bf u} \ra_{\mathbb{F}_{q^n}} \colon {\bf u}\in U\setminus \{{\bf 0} \}\}.\]
One of the most natural questions about linear sets is their equivalence. Two linear sets $L_U$ and $L_V$ of $\PG(1,q^n)$ are said to be \emph{$\mathrm{P\Gamma L}$-equivalent} (or simply \emph{equivalent}) if there is an element in $\mathrm{P\Gamma L}(2,q^n)$ mapping $L_U$ to $L_V$. In the applications it is crucial to have methods to decide whether two linear sets are equivalent or not. This can be a difficult problem and some results in this direction can be found in \cite{CSZ2015, CSMP2016}.
 If $L_U$ and $L_V$ are two equivalent $\F_q$-linear sets of rank $n$ in $\PG(1,q^n)$ and $\varphi$ is an element of $\G(2,q^n)$ which induces a collineation mapping $L_U$ to $L_V$, then $L_{U^\varphi}=L_V$. Hence the first step to face with the equivalence problem for linear sets is to determine which $\F_q$-subspaces can define the same linear set. 
 
For any $q$-polynomial $f(x)=\sum_{i=0}^{n-1} a_i x^{q^i}$ over $\F_{q^n}$, the graph ${\mathcal G}_f=\{(x,f(x)) \colon x\in \F_{q^n}\}$ is an $\F_q$-vector subspace of the 2-dimensional vector space $V=\F_{q^n}\times\F_{q^n}$ and the point set
\[L_f:=L_{{\mathcal G}_f}=\{\la (x,f(x))\ra_{\mathbb{F}_{q^n}}  \colon x\in \F_{q^n}^*\}\] is an $\F_q$-linear set of rank $n$ of $\PG(1,q^n)$.
 In this context, the problem posed in \eqref{start} corresponds to find all $\F_q$-subspaces of $V$ of rank $n$ (cf. \cite[Proposition 2.3]{CSMP2016}) defining the linear set $L_f$. The maximum field of linearity of $f$ is the maximum field of linearity of $L_f$, and it is well-defined (cf. Proposition \ref{fieldoflinearity} and  \cite[Proposition 2.3]{CSMP2016}). Also, by the Introduction from any $q$-polynomial $f$ over $\F_{q^n}$, the linear sets $L_f$, $L_{f_\lambda}$ (with $f_\lambda(x):=f(\lambda x)/\lambda$ for each $\lambda\in\F_{q^n}^*$) and $L_{\hat f}$ coincide (cf. \cite[Lemma 2.6]{BGMP2015} and the first part of \cite[Section 3]{CSMP2016}). 
 If $f$ and $g$ are two equivalent $q$-polynomials over $\F_{q^n}$, i.e. ${\mathcal G}_f$ an ${\mathcal G}_g$ are equivalent w.r.t. the action of the group $\Gamma{\mathrm L}(2,q^n)$, then the corresponding $\F_q$-linear sets $L_f$ and $L_g$ of $\PG(1,q^n)$ are ${\mathrm P}\Gamma{\mathrm L}(2,q^n)$-equivalent. The converse does not hold (see  \cite{CSZ2015} and \cite{CSMP2016} for further details). More precisely,
\begin{proposition}
\label{pgamma-equiv}
Let $L_f$ and $L_g$ be two $\F_q$-linear sets of rank $n$ of $\PG(1,q^n)$. Then $L_f$ and $L_g$ are $\mathrm{P\Gamma L}(2,q^n)$-equivalent if and only if there exists and element $\varphi\in\Gamma\mathrm{L}(2,q^n)$ such that $Im\,(f_\varphi(x)/x)=Im\,(g(x)/x)$.
\qed
 \end{proposition}

Linear sets of rank $n$ of $\PG(1,q^n)$ have size at most $(q^n-1)/(q-1)$.
A linear set $L_U$ of rank $n$ whose size achieves this bound is called \emph{maximum scattered}.
For applications of these objects we refer to \cite{OP2010} and \cite{Lavrauw}. 

Following \cite{LuMaPoTr2014} and \cite{DoDu9} a maximum scattered $\F_q$-linear set $L_U$ of rank $n$ in $\PG(1,q^n)$ is of \emph{pseudoregulus type} if 
it is $\mathrm{P}\Gamma\mathrm{L}(2,q^n)$-equivalent to $L_f$ with $f(x)=x^q$ or, equivalently, if there exists $\varphi\in\mathrm{GL}(2,q^n)$ such that
\[L_{U^\varphi}=\{\la(x,x^q)\ra_{\F_{q^n}}\colon x\in\F_{q^n}^*\}.\]
By Proposition \ref{pgamma-equiv} and Corollary \ref{norm}, it follows
\begin{proposition}
\label{peseud-equiv}
An $\F_q$-linear set $L_f$ of rank $n$ of $\PG(1,q^n)$ is
% ${\mathrm P}\Gamma {\mathrm L}(2,q^n)$-equivalent to a linear set 
of pseudoregulus type if and only if $f(x)$ is equivalent to $x^{q^i}$ for some $i$ with $\gcd(i,n)=1$.
\qed
\end{proposition}
For the proof of the previous result see also \cite{LaShZa2013}.

The known pairwise non-equivalent families of $q$-polynomials over $\F_{q^n}$ which define maximum scattered linear sets of rank $n$ in $\PG(1,q^n)$ are

\begin{enumerate}
\item $f_{s}(x)= x^{q^s}$, $1\leq s\leq n-1$, $\gcd(s,n)=1$ (\cite{BL2000,CSZ2016}),
\item $g_{s,\delta}(x)= \delta x^{q^s} + x^{q^{n-s}}$, $n\geq 4$, $\N_{q^n/q}(\delta)\notin \{0,1\}$ \footnote{This condition  implies $q\neq 2$.}, $\gcd(s,n)=1$ (\cite{LP2001} for $s=1$, \cite{Sh,LTZ} for $s\neq 1$),
\item $h_{s,\delta}(x):= \delta x^{q^s}+x^{q^{s+n/2}}$, $n\in \{6,8\}$, $\gcd(s,n/2)=1$, $\N_{q^n/q^{n/2}}(\delta) \notin \{0,1\}$, for the precise conditions on $\delta$ and $q$ see \cite[Theorems 7.1 and 7.2]{CMPZ} \footnote{Also here $q>2$, otherwise the linear set defined by $h_{s,\delta}$ is never scattered.},
\item $k_b(x):=x^q+x^{q^3}+bx^{q^5}$, $n=6$, with $b^2+b=1$, $q\equiv 0,\pm 1 \pmod 5$ (\cite{CSMZ2017}).
\end{enumerate}

\begin{remark}
All the previous polynomials in cases 2.,3.,4. above are examples of functions which are not equivalent to monomials but the set of directions determined by their graph has size $(q^n-1)/(q-1)$, i.e. the corresponding linear sets are maximum scattered. The existence of such linearized polynomials is briefly discussed also in \cite[p.\ 132]{Praha}.
\end{remark}

For $n=2$ the maximum scattered $\F_q$-linear sets coincide with the Baer sublines. For $n=3$ the maximum scattered linear sets  are all of pseudoregulus type and the corresponding $q$-polynomials are all $\mathrm{GL}(2,q^3)$-equivalent to $x^q$  (cf. \cite{LaVa2010}). For $n=4$ there are two families of maximum scattered linear sets. More precisely, if $L_f$ is a maximum scattered linear set of rank $4$ of $\PG(1,q^4)$, with maximum field of linearity $\F_{q}$, then there exists $\varphi\in\mathrm{GL}(2,q^4)$ such that either $f_\varphi(x)=x^q$ or $f_\varphi(x)=\delta x^q+x^{q^3}$, for some $\delta\in\F_{q^4}^*$ with $\N_{q^4/q}(\delta)\notin\{0,1\}$ (cf. \cite{CSZ2017}). It is easy to see that $L_{f_{1}}=L_{f_s}$ for any $s$ with $\gcd(s,n)=1$, and $f_i$ is equivalent to $f_j$ if and only if $j\in\{i,n-i\}$. Also, the graph of $g_{s,\delta}$ is  $\mathrm{GL}(2,q^n)$-equivalent to the graph of $g_{n-s,\delta^{-1}}$.

In \cite[Theorem 3]{LP2001} Lunardon and Polverino proved that $L_{g_{1,\delta}}$ and $L_{f_1}$ are not $\mathrm{P}\Gamma \mathrm{L}(2,q^n)$-equivalent when $q>3$, $n\geq 4$. This was extended also for $q=3$ \cite[Theorem 3.4]{CSMZ2017}. %For $n=5$, in \cite{CMPxxxx} it is proved that  $L^{2,5}_{2,\delta}$ is $\mathrm{P}\Gamma\mathrm{L}(2,q^5)$-equivalent neither to $L^{2,5}_{1,\delta'}$ nor to $L^{1,5}_1$. 
Also in \cite{CSMZ2017}, it has been proven that  for $n=6,8$ the linear sets $L_{f_1}$, $L_{g_{s,\delta}}$, $L_{h_{s',\delta'}}$ and $L_{k_b}$ are pairwise non-equivalent for any choice of $s,s',\delta, \delta',b$.

%For $n=5$ the known inequivalent $q$-polynomials over $\F_{q^5}$, with maximum field of linearity $\F_q$, are 
%%maximum scattered $\F_q$-subspaces are $U_f$, $U_g$, $U_h$ with
%\begin{enumerate}
%	\item $f_i(x)=x^{q^i}$, $i\in \{1,2\}$ ($L_f$ is of pseudoregulus type),
%	\item $g(x)=\mu x^q+x^{q^4}$ with $\N(\mu)\notin \{0,1\}$ ($L_g$ is of Lunardon-Polverino type),
%	\item $h(x)=\delta x^{q^2}+x^{q^3}$ with $\N(\delta)\notin \{0,1\}$ (found by Sheekey).
%\end{enumerate}
%
%
%
%The subspace ${\mathcal G}_h$ is a generalization of ${\mathcal G}_g$, for the general family we refer to \cite{Sh}.
%In \cite{LP2001} it was proved that for $n>3$ and $q>3$ the linear set $L_g$ is not of pseudoregulus type, i.e. $g$ is not equivalent to a monomial function. This was extended also for $q=3$ and for the whole family found by Sheekey and hence also for the members of $L_h$ \cite[Theorem 3.4]{CSMZ2017}. In this section we prove that one can find for each $q>2$ a suitable $\delta$ such that $L_h$ is not equivalent (w.r.t. $\mathrm{P}\Gamma\mathrm{L}(2,q^5)$) to the linear sets found by Lunardon and Polverino. In order to this, we first reformulate Theorem \ref{main} as follows.

In this section we prove that one can find for each $q>2$ a suitable $\delta$ such that $L_{g_{2,\delta}}$ of $\PG(1,q^5)$ is not equivalent to the linear sets $L_{g_{1,\mu}}$ of $\PG(1,q^5)$ for each $\mu\in\F_{q^5}^*$, with $\N_{q^5/q}(\mu)\notin\{0,1\}$. In order to do this, we first reformulate Theorem \ref{main} as follows.
%
%\begin{theorem}[Theorem \ref{main}]
%\label{main1}
%	Let $f(x)$ and $g(x)$ be two $q$-polynomials over $\F_{q^5}$ such that $L_f=L_g$. Then either  $L_f=L_g$ is of pseudoregulus type or there exists $\lambda\in \F_{q^5}^*$ such that one of $\lambda {\mathcal G}_g={\mathcal G}_f$ or $\lambda {\mathcal G}_g={\mathcal G}_{\hat f}$ holds.
%\end{theorem}

\begin{theorem}[Theorem \ref{main}]
\label{main1}
	Let $f(x)$ and $g(x)$ be two $q$-polynomials over $\F_{q^5}$ such that $L_f=L_g$. Then either  $L_f=L_g$ is of pseudoregulus type or there exists $\lambda\in \F_{q^5}^*$ such that $g(x)=f(\lambda x)/\lambda$ or $g(x)=\hat{f}(\lambda x)/\lambda$ holds.
\end{theorem}

\begin{theorem}
Let $g_{2,\delta}(x)=\delta x^{q^2}+x^{q^3}$ for some $\delta\in \F_{q^5}^*$ with $\N(\delta)^5 \neq 1$. Then $L_{g_{2,\delta}}$ is not ${\mathrm P}\Gamma{\mathrm L}(2,q^5)$-equivalent to any linear set $L_{g_{1,\mu}}$ and hence it is a new a maximum scattered linear set.
\end{theorem}
\begin{proof}
%	It follows from Theorem \ref{main} that the $\G$-class of $L_g$ is at most two.
	Suppose, contrary to our claim, that  $L_{g_{2,\delta}}$ is ${\mathrm P}\Gamma{\mathrm L}(2,q^5)$-equivalent to a linear set $L_{g_{1,\mu}}$. From Proposition \ref{pgamma-equiv} and Theorem \ref{main1}, taking into account that $L_{g_{1,\mu}}$ is not of pseudoregulus type, it follows that there exist $\varphi\in\Gamma{\mathrm L}(2,q^5)$ and $\lambda\in\F_{q^5}^*$ such that either $(g_{2,\delta})_{\varphi}(x)=g_{1,\mu}(\lambda x)/\lambda$ or $(g_{2,\delta})_{\varphi}(x)=\hat g_{1,\mu}(\lambda x)/\lambda$. This is equivalent to say that 
	there exist $A,B,C,D \in \F_{q^5}$ with $AD-BC \neq 0$ and a field automorphism $\tau$ of $\F_{q^5}$
%	 the associated field automorphism is $x \mapsto x^\tau$ for some $\tau=p^k$,
such that
\[
\left\{
\begin{pmatrix}
A & B \\
C & D \\
\end{pmatrix}
\begin{pmatrix}
x^\tau \\
h(x)^\tau \\
\end{pmatrix}
\colon x\in \F_{q^5}
\right\}
=
\left\{
\begin{pmatrix}
z\\
\alpha z^q + \beta z^{q^4}\\
\end{pmatrix}
\colon z\in \F_{q^5}
\right\}
,\]
where  $N(\alpha) \neq N(\beta)$ and $\alpha \beta \neq 0$.
%Then
%	\[g(Ax^{\tau}+Bf(x)^{\tau})=Cx^{\tau}+Df(x)^{\tau},\]
%	for each $x\in \F_{q^5}$.
	We may substitute $x^{\tau}$ by $y$, then
	\[\alpha (Ay+B\delta^\tau y^{q^2}+By^{q^3})^q+\beta (Ay+B\delta^\tau y^{q^2}+By^{q^3})^{q^4} = \]
	\[ Cy+D\delta^\tau y^{q^2}+Dy^{q^3}\]
for each $y\in \F_{q^5}$.
%	\[\alpha A^q y^q +\alpha B^q \delta^{q\tau} y^{q^3}+\alpha B^q y^{q^4}+\beta A^{q^4}y^{q^4}+\beta B^{q^4}\delta^{q^4 \tau} y^q+\beta B^{q^4}y^{q^2} = \]
%	\[ Cy+D\delta^\tau y^{q^2}+Dy^{q^3}. \]
	Comparing coefficients yields $C=0$ and
	\begin{equation}
	\label{1}
	\alpha A^q+\beta B^{q^4}\delta^{q^4 \tau}=0,
	\end{equation}
	\begin{equation}
	\label{2}
	\beta B^{q^4}=D\delta^\tau,
	\end{equation}
	\begin{equation}
	\label{3}
	\alpha B^q \delta^{q\tau} =D,
	\end{equation}
	\begin{equation}
	\label{4}
	\alpha B^q +\beta A^{q^4}=0.
	\end{equation}
	Conditions \eqref{2} and \eqref{3} give
	\begin{equation}
	\label{primo}
	B^{q^4-q}= \delta^{(q+1)\tau} \alpha/\beta.
	\end{equation}
	On the other hand from \eqref{4} we get $A^q=- B^{q^3}\alpha^{q^2}/\beta^{q^2}$ and substituting this into \eqref{1} we have
%	\begin{equation}
%	B^{q^3}\alpha^{q^2}/\beta^{q^2}=\beta B^{q^4}\delta^{q^4 \tau}/\alpha,
%	\end{equation}
%	and hence
	\begin{equation}
	\label{secondo}
	B^{q^3-q^4}=\delta^{q^4 \tau} \beta^{q^2+1}/\alpha^{q^2+1}.
	\end{equation}
	Equations \eqref{primo} and \eqref{secondo} give $N(\beta/\alpha)=N(\delta)^{2\tau}$ and  $N(\alpha/\beta)^2=N(\delta)^{\tau}$, respectively. It follows that
	$N(\delta)^{5\tau}=1$ and hence $N(\delta)^5=1$, a contradiction.
\end{proof}

\section*{Open problems}

We conclude the paper by the following open problems.

\begin{enumerate}
	\item Is it true also for $n>5$ that for any pair of $q$-polynomials $f(x)$ and $g(x)$ of $\F_{q^n}[x]$, with maximum field of linearity $\F_q$, if $Im\,(f(x)/x)=Im\,(g(x)/x)$ then either there exists $\varphi\in \Gamma\mathrm{L}(2,q^n)$ such that $f_{\varphi}(x)=\alpha x^{q^i}$ and $g_{\varphi}(x)=\beta x^{q^j}$ with $\N(\alpha)=\N(\beta)$ and $\gcd(i,n)=\gcd(j,n)=1$, or there exists $\lambda\in \F_{q^n}^*$ such that $g(x)=f(\lambda x)/\lambda$ or $g(x)=\hat f(\lambda x)/\lambda$?
	\item Is it possible, at least for small values of $n>4$, to classify, up to equivalence, the $q$-polynomials $f(x)\in \F_{q^n}[x]$ such that $|Im\,(f(x)/x)|=(q^n-1)/(q-1)$? Find new examples!
	\item Is it possible, at least for small values of $n$, to classify, up to equivalence, the $q$-polynomials $f(x)\in \F_{q^n}[x]$ such that  $|Im\,(f(x)/x)|=q^{n-1}+1$? Find new examples!
	\item Is it possible, at least for small values of $n$, to classify, up to equivalence, the $q$-polynomials $f(x)\in \F_{q^n}[x]$ such that in the multiset  $\{f(x)/x \colon x\in \F_{q^n}^*\}$ there is a unique element which is represented more than $q-1$ times? In this case the linear set $L_f$ is an \emph{$i$-club} of rank $n$ and when $q=2$, then such linear sets correspond to translation KM-arcs cf. \cite{KMarcs2} (a KM-arc, or $(q+t)$-arc of type $(0,2,t)$, is a set of $q+t$ points of $\PG(2,2^n)$, such that each line meet the point set in 0,2 or in $t$ points, cf. \cite{KMarcs}). Find new examples!
	\item Determine the equivalence classes of the set of $q$-polynomials in $\F_{q^4}[x]$.
	\item Determine, at least for small values of $n$, all the possible sizes of $Im\,(f(x)/x)$ where $f(x)\in \F_{q^n}[x]$ is a $q$-polynomial.
	
\end{enumerate}

\noindent Bence Csajb\'ok\\
MTA--ELTE Geometric and Algebraic Combinatorics Research Group\\
ELTE E\"otv\"os Lor\'and University, Budapest, Hungary\\
Department of Geometry\\
1117 Budapest, P\'azm\'any P.\ stny.\ 1/C, Hungary\\
{{\em csajbokb@cs.elte.hu}}
\bigskip

\noindent Giuseppe Marino, Olga Polverino\\
Dipartimento di Matematica e Fisica,\\
Universit\`a degli Studi della Campania ``Luigi Vanvitelli'',\\
Viale Lincoln 5, I-\,81100 Caserta, Italy\\
{{\em giuseppe.marino@unicampania.it}, {\em olga.polverino@unicampania.it}}
\bigskip

\end{document}